\documentclass[twoside,11pt]{amsart}
\usepackage{amsmath}
\usepackage{amssymb}
\usepackage{latexsym}
\usepackage{verbatim}
\usepackage{graphics}
\usepackage{graphicx}
\newtheorem{theorem}{Theorem}[section]
\newtheorem{corollary}[theorem]{Corollary}
\newtheorem{lemma}[theorem]{Lemma}
\newtheorem{proposition}[theorem]{Proposition}

\newtheorem{definition}[theorem]{Definition}
\newtheorem{remark}[theorem]{Remark}
\newtheorem{example}[theorem]{Example}
\numberwithin{equation}{section}

\def\b0{{\bf 0}}
\def\bv{{\bf v}}
\def\by{{\bf y}}
\def\bg{{\bf g}}
\def\bx{{\bf x}}
\def\bt{{\bf t}}
\def\bz{{\bf z}}
\def\bw{{\bf w}}
\def\bu{{\bf u}}
\def\ba{{\bf a}}
\def\bB{{\bf b}}
\def\AJC{A_{\J_i^c}}

\newcommand{\ep}{\varepsilon}

\def\RR{{\mathbb R}}
\def\ZZ{{\mathbb Z}}
\def\Z{{\mathbb Z}}

\def\NN{{\mathbb N}}
\def\Nat{\NN}

\def\ZZ{{\mathbb Z}}

\def\P{{\mathcal P}}

\def\S{{\mathcal S}}
\def\F{{\mathcal F}}
\def\B{{\mathcal B}}
 
 \def\M{{\mathcal M}}
\def\A{{\mathcal A}}
\def\T{{\mathcal T}}

\def\R{{\RR}}
\def\Tk{\T}
\def\Sk{\S}

\def\dist{{\rm dist}}

\def\freq{{\rm freq}}
\def\J{{\mathcal J}}
\def\I{{\mathcal I}}

\def\supp{{\rm supp}}
\def\diam{{\rm diam}}
\def\vol{{\rm vol}}

\def\clos{{\rm Clos}}
\def\lam{\lambda}
\def\sig{\sigma}
\def\om{\omega}
\def\gam{\gamma}
\def\Anonp{\A_{\rm nonp}}
\def\Fnonp{\F_{\rm nonp}}
\def\Xnonp{X_{\rm nonp}}
\def\per{{\rm per}}
\def\Aper{\A_{\rm per}}
\def\Int{{\rm int}}
\def\es{\emptyset}

\def\seq{\subseteq}
\def\wt{\widetilde}
\def\Ak{\A}

\def\Xa{X_{\A,\om}}

\def\ov{\overline}
\def\Gam{\Gamma}
\def\be{\begin{equation}}
\def\ee{\end{equation}}
\def\Th{\Theta}
\def\Bl{B_\eta(\b0)}
\begin{document}
\title {Invariant measures for non-primitive tiling substitutions}
\author{M. I. Cortez}
\address{Mar\'{\i}a  Isabel Cortez, Departamento de Matem\'atica y Ciencia de la Computaci\'on, Universidad de Santiago de Chile, Av. Libertador Bernardo O'Higgins 3363, Santiago, Chile.}
 \email{maria.cortez@usach.cl}
%\thanks{$\dag$Partially supported by Fondecyt 1100318}
\author{B. Solomyak}
\address{Boris Solomyak, Box 354350, Department of Mathematics,
University of Washington, Seattle WA 98195}
\email{solomyak@math.washington.edu}
\thanks{M. I. Cortez acknowledges financial support from proyecto Fondecyt
1100318. B. Solomyak was partially supported by NSF grant
DMS-0654408.}
\begin{abstract}
We consider self-affine tiling substitutions
in Euclidean space and the corresponding tiling dynamical systems. It is well-known that in the primitive case the
dynamical system is uniquely ergodic. We investigate invariant measures when the substitution is not primitive, and
the tiling dynamical system is non-minimal. We prove that all ergodic invariant
probability measures are supported on minimal components, but there are other natural ergodic invariant measures, which are infinite.
Under some mild assumptions, we completely characterize $\sig$-finite invariant measures which are positive and finite on a cylinder set.
A key step is to establish recognizability of non-periodic tilings in our setting.
Examples include the ``integer Sierpi\'nski gasket and carpet'' tilings. For
such tilings the only invariant probability measure is supported on trivial periodic tilings, but there is a fully
supported $\sig$-finite invariant measure, which is locally finite and unique up to scaling.
\end{abstract}
\maketitle
\section{Introduction}\label{introduction}
We consider self-affine substitution tilings of $\R^d$. A tile is a compact subset of $\R^d$ which is a closure of its interior.
A tiling is a set of tiles with disjoint interiors whose union is all of $\R^d$. We restrict ourselves to tilings which satisfy the
(translational) {\em finite pattern condition}, abbreviated FPC, i.e.\ for any $R>0$,
there are finitely many patterns, or patches, of diameter less $R$, up to translation. In particular, there are finitely many tiles up to translation;
we call some representatives of equivalence classes the {\em prototiles}. Sometimes tiles of the same shape need to be distinguished; this is achieved by
considering a tile as a pair $T=(F,j)$ where $F=\supp(T)$ is a compact set (the {\em support} of the tile), and $j\in \{1,\ldots,\ell\}$ for some $\ell\ge 1$,
is a {\em label} (or ``type'', or ``color'') of the tile. When translating a tile or a patch, the labels of the tiles are preserved.
Let $\Ak$ be a finite set of prototiles and $\Ak^+$ the set of
patches whose every tile is a translate of some $A\in \Ak$. Given
an expansive linear map $\varphi:\,\R^d\to \R^d$, we say that
$\om:\,\Ak \to \Ak^+$ is a {\em tile substitution with expansion}
$\varphi$ if the union of tiles in $\om(A)$ equals
$\varphi(\supp(A))$. In other words, every ``inflated tile'' can
be subdivided into translates of the prototiles. This property
allows us to iterate the substitution and obtain a family of
patches $\om^n(A),\, A\in \Ak,\, n\ge 1$. The {\em substitution
tiling space} $\Xa$ is defined as the set of all tilings of $\R^d$
whose every patch is a translate of a subpatch of $\om^n(A)$ for
some $A\in \Ak$ and $n\in \Nat$. We assume that this space has the
FPC property. We also assume that every prototile $A\in \Ak$ is
{\em admissible}, that is, there exists a tiling $\Tk\in \Xa$ with
$A\in \Tk$.
The space $\Xa$ is compact in the usual ``local'' metric, in which two tilings are considered to be close if they agree on a large ball around the origin up to
a small translation (see the next section for precise definitions), and $\R^d$ acts on $\Xa$ continuously by translations. This is the
{\em tiling dynamical system} associated with $\om$.
To a tiling substitution $\om$ we associate the {\em substitution matrix} $M$ whose entry $M(i,j)$ equals the number of tiles of type $i$ which appear in
the substitution $\om$ applied to a tile of type $j$. The substitution is {\em primitive} if $M^k>0$ for some $k\in \Nat$. Substitution tiling dynamical systems
have been studied almost exclusively in the primitive case, when they are
minimal and uniquely
ergodic. Here we begin the investigation of non-primitive tiling substitutions.
A tiling substitution can be viewed as a generalization of a symbolic substitutions, see \cite{Queff,Pytheas,Robi}.
Recently, a systematic investigation of non-primitive (one-dimensional) symbolic substitutions has been started in \cite{BKM,BKMS} (see also
\cite{Fi1,Yuasa1,Yuasa2}). Our work builds on \cite{BKMS}; however, we have to introduce many new ingredients.
The substitution
matrix $M$ is non-negative, and we can consider its irreducible components, see \cite[4.4]{LM}. We will show that $\Xa = X_{\Ak,\om^k}$ for $k\in \Nat$;
thus, by raising the
substitution to a positive power we can assume, without loss of generality that all irreducible components are primitive or equal to $[0]$.
Suppose that $M$ has irreducible components $M_1,\ldots,M_\ell$, and let $\A_1,\ldots,A_\ell$ be the corresponding subsets of the prototile set.
By reordering the prototiles, we
can assume that $M$ has a block upper-triangular form, with the diagonal blocks $M_i,\,i\le \ell$. In terms of the substitution, this means that $\om(A),\ A\in \Ak_j$,
contains only translated of the tiles from $\bigcup_{i=1}^j \Ak_i$.
Let $m\ge 1$ be the number of ``minimal'' irreducible components having the property that $\om(A) \seq \A_i^+$ for all $A\in \A_i$.
Then $\om_i:=\om|_{\Ak_i}$ is the usual primitive substitution for $i\le m$. It turns out that $X_i:=X_{\Ak_i,\om_i}$, $i\le m$, are precisely the minimal components of
the tiling dynamical system.
Our first main result is the following.

\medskip

\noindent {\bf Theorem A.} {\em All ergodic invariant probability measures for the substitution tiling system are supported on minimal components.}

\medskip

The proof uses the pointwise ergodic theorem and the fact that
only the patches from minimal components may have a positive
frequency in a tiling. However, this is only the beginning of the
story, as there are natural and interesting {\em infinite}
invariant measures. In order to characterize them, we need the
property of {\em recognizability}, namely, the invertibility of
the substitution map $\om$ extended to the tiling space $\Xa$. We
prove that it holds whenever the tiling substitution is
non-periodic, namely $\Tk+\bv\ne \Tk$ for $\Tk\in \Xa$ and $\bv
\ne \b0$, generalizing \cite{So} from the primitive case. Even
more, we establish recognizability of non-periodic tilings when
the tiling space does contain periodic ones, under a mild
geometric condition (the ``non-periodic border'' condition, see
Section 4 for details). Applying recognizability, we construct a
sequence of {\em nested Kakutani-Rokhlin partitions} of the
transversal which allows us to determine the natural $\sig$-finite
measures, first on the transversal, and then on the tiling space.
(Actually, in general it is only a covering, but it becomes a
partition when restricted to the set of non-periodic tilings,
which is enough for our purposes.) The {\em transversal} is
defined as the set of tilings which have a tile with a
``puncture'' at the origin; this is consistent with a view of the
tiling space as a lamination. {\em Transverse measures} are in
1-to-1 correspondence with invariant measures. Although the use of
transverse measures in tiling dynamics is by now standard, see
\cite{BBG}, we have to extend the theory to our setting, namely,
to the non-primitive case, and to include $\sigma$-finite
measures; this is done in the Appendix. In order to state our
results on $\sig$-finite measures, we need to introduce some
terminology. There are irreducible components $M_i$ and the
corresponding prototile subsets $\Ak_i$, which we call
``maximal'': they are characterized by the property {\em
$$
\om(A)\ \ \mbox{contains a tile of type $\Ak_i$}\ \ \Longrightarrow\ \ A\in \Ak_i.
$$
}
Let $p \ge 1$ be the number of maximal components; they correspond to prototile subsets $\Ak_{\ell-p+1},\ldots,$ $\Ak_\ell$.
Denote by $Y_i$, for $i\ge \ell-p+1$, the set of tilings $\Tk\in \Xa$ which contain at least one tile of type $\Ak_i$. The subsets
$Y_i\seq \Xa$ are non-empty (by the admissibility assumption), open, and invariant. We will show that $(Y_i,\R^d)$ is a
non-compact minimal tiling dynamical system for $i\ge \ell-p+1$, $Y_i\cap Y_j=\es$ for $i\ne j$, and $\bigcup_{i=1}^p Y_i$ is dense in $\Xa$.
We call $Y_i$ the {\em maximal} components of the tiling dynamical system.
Now we can state our second main result.

\medskip
\noindent
{\bf Theorem B.} {\em Suppose that the tiling substitution satisfies the non-periodic border condition (in particular, it is satisfied if
the substitution is non-periodic). Then for each $i=\ell-p+1,\ldots,\ell$, there is a unique, up to scaling,
ergodic invariant measure supported on $Y_i$, such that every point has an open neighborhood
with positive finite measure.}

\medskip

{\bf Remarks.} 1. It is often the case that $\Xa \ne \bigcup_{i=1}^m X_i \cup
\bigcup_{i=\ell-p+1}^\ell Y_i$, and there are other infinite
invariant measures, but those described in Theorem B
are the most natural ones. In Section 5
we classify all ergodic
invariant measures which are positive and finite on a ``cylinder
set''; they correspond to some irreducible components of
$M$.

\indent
2. There is, in general, a greater variety of invariant measures for (one-dimensional) symbolic substitutions than for tile substitutions considered
here; in particular, there are sometimes ergodic invariant probability measures of full support, see \cite{BKMS}. The reason is that in our case
the vector of volumes of the prototiles is always a strictly positive left eigenvector of the substitution matrix, whereas for a symbolic substitution such
an eigenvector may fail to exist.

\medskip

\indent The paper is organized as follows. Section 2 contains preliminaries, including the topological results about minimal and maximal components.
Theorem A is proved in Section 3. We obtain recognizability results, which are of independent interest,
in Section 4, and in Section 5 we investigate $\sig$-finite invariant measures and prove Theorem B. (We would like to point out that Sections 4 and 5 can be read independently.) Section 6 is devoted to examples and concluding
remarks. For specific examples the recognizability
properties can usually be checked directly. In Section 7 (the Appendix) we treat transverse measures.
\medskip

%We illustrate our results with two examples. Additional details about them and other examples are discussed in Section 6.

\medskip

{\bf Acknowledgment.} We are grateful to Karl Petersen for his interest, helpful comments, and suggestion to consider the  ``Sierpi\'nski carpet'' and
``gasket'' tilings.

%%%%%%%%%%%%%%%%%%%%%%%%%%%%%%%%%%%%%%%%%%%%%%%%%%%%%%%%%%%%%%%%%%%%%%%%%%%%%%%%%%%%%%%%%%%%%%%%%%%%%%%%%%%%%%%%%%%%
\section{Preliminaries}
\label{preliminaries}
\subsection{Tilings.}
Fix a set of ``types'' labeled by $\{1,\ldots,N\}$, with $N\ge 1$.
A {\em tile} in $\RR^d$ is defined as a pair $T = (F,i)$ where
$F  = \supp(T)$ (the support of $T$) is a compact set in $\R^d$ which is the
closure of its interior, and $i = \ell(T)\in \{1,\ldots,N\}$ is the type of $T$.
A {\em tiling} of $\R^d$ is a set $\Tk$ of tiles such that $\R^d = \bigcup
\{\supp(T):\ T\in \Tk\}$ and distinct tiles (or rather, their supports)
have disjoint interiors.
A patch $P$ is a finite set of tiles with disjoint interiors. The
{\em support of a patch} $P$ is defined by $\supp (P) =
\bigcup\{\supp(T):\ T\in P\}$. The {\em diameter of a patch} $P$
is $\diam(P)=\diam(\supp(P))$. The {\em translate} of a tile
$T=(F,i)$ by a vector $\bg\in \R^d$ is $T+\bg =
(F+\bg, i)$. The translate of a patch $P$ is $P+\bg =
\{T+\bg:\ T\in P\}$. We say that two patches $P_1,P_2$ are
{\em translationally equivalent} if $P_2 = P_1+\bg$ for some
$\bg\in \R^d$. Finite subsets of $\Tk$ are called
$\Tk$-patches.
For a set $B\seq \RR^d$ we write
$$
[B]^\Tk = \{T\in \Tk:\ \supp(T)\cap B \ne \es\}.
$$
%We similarly define $[B]^P$ for a patch $P$.
\begin{definition} \label{def-fpc} {
We say that a tiling $\Tk$ has (translational) {\em finite patch complexity} (FPC), or satisfies the {\em finite pattern condition},
if for any $R>0$ there are finitely many $\Tk$-patches of diameter less
than $R$ up to translation equivalence. This definition naturally extends to any collection of tilings.
}
\end{definition}
\begin{definition} \label{def-rep}
{A tiling $\Tk$ is {\em repetitive} if for any patch
$P\subset \Tk$ there is $R>0$ such that for any $\bx\in \R^d$
there is a $\Tk$-patch $P'$ such that $\supp(P')\subset
B_R(\bx)$ and $P'$ is a translate of $P$. }
\end{definition}
\subsection{Tile substitutions, self-affine tilings.}
A linear map $\varphi : \R^d \rightarrow \R^d$ is {\em expansive}
if all its eigenvalues lie outside the unit circle.
\begin{definition}\label{def-subst}
{Let $\Ak = \{A_1,\ldots,A_N\}$ be a finite set of tiles in $\R^d$
such that for $i\neq j$ the tiles $A_i$ and $A_j$ are not
translationally equivalent; we will call them {\em prototiles}. We
assume that every prototile is ``centered at the origin'', in the
sense that $\b0\in \Int(\supp(A_j))$ for all $j$. Denote by
$\Ak^+$ the set of patches made of tiles each of which is a
translate of one of the prototiles. A map $\omega: \Ak \to \Ak^+$
is called a {\em tile substitution} with expansion $\varphi$ if
\begin{equation} \label{def-sub}
\supp(\om(A_j)) = \varphi (\supp(A_j)) \ \ \  \mbox{for} \  j\le N.
\end{equation}
}
\end{definition}
In plain language, every expanded prototile $\varphi (A_j)$ can be decomposed into
a union of tiles (which are all translates of the prototiles) with disjoint
interiors.
The substitution $\om$ is extended to all translates of prototiles
by $\om(\bx+A_j)= \varphi(\bx) + \om(A_j)$, and to patches
by $\om(P)=\bigcup\{\om(T):\ T\in P\}$. This is well-defined due to
(\ref{def-sub}).
Denote by $X_\Ak$ the set of tilings whose tiles belong to $\Ak$ up to translation; note that $\om$ acts on $X_\Ak$ as well.
%The substitution $\om$ acts on $X_\Ak$ as follows: given $\Tk\in X_\Ak$, we first expand every tile by $\varphi$, and then
%apply the subdivision rule implicit in (\ref{def-sub}).
\begin{definition} \label{def-tsp} {
Let $\om$ be a tile substitution. A patch $P$ is said to be {\em
legal} if there exists $n\ge 1$, $A_j\in \Ak$, and $\bx\in
\R^d$, such that $P+\bx \subseteq \om^n(A_j)$. Denote by
$\Xa\seq X_\A$  the set of all tilings of $\R^d$ whose every
patch is legal. The set $X_{\Ak,\om}$ is called the {\em tiling
space} corresponding to the substitution.
We say that the substitution $\om$ has FPC if the space $X_{\Ak,\om}$ has FPC.
The substitution $\om$ is {\em admissible} if for every prototile $A_j$ there exists $\Tk\in \Xa$ such that $A_j\in \Tk$.
}
\end{definition}
The additive group $\R^d$ acts on $X_{\Ak,\om}$ by translations;
this action $(X_{\Ak,\om},\R^d)$ is called the {\em tiling
dynamical system} or the {\em self-affine tiling dynamical system}
associated to $\om$. It is clear from the definitions that
$\om(\Xa)\seq\Xa$.

We use a {\em tiling metric} on $\Xa$, in which
two tilings are close if after a small translation they agree
on a large ball around the origin. To make it precise,
we say that two tilings $\Tk_1,\Tk_2$ {\em agree} on a set $K \subset \R^d$ if
$$
\supp(\Tk_1\cap \Tk_2) \supseteq K.
$$
For $\Tk_1,\Tk_2 \in \Xa$ let
$$
\widetilde{\varrho}(\Tk_1,\Tk_2) := \inf\{r \in (0,2^{-1/2}): \
\exists\,\bg,\ \|\bg\| \le r\ \mbox{such that}
$$
$$
\Tk_1-\bg \mbox{ \ agrees with $\Tk_2$ on $B_{1/r}(\b0)$} \}.
$$
Then
$$
\varrho(\Tk_1,\Tk_2): = \min\{2^{-1/2},\widetilde{\varrho}(\Tk_1,\Tk_2)\}.
$$

\begin{theorem}\cite{Rud} {\em (see also \cite{Robi}).}
$(\Xa,\varrho)$ is a complete metric space. It is compact,
whenever the space has FPC. The action of $\R^d$ by translations
on $\Xa$, given by  $\bg(\Sk)= \Sk-\bg$, is continuous.
\end{theorem}
\begin{definition} \label{def-submat} {
To the tile substitution $\om$ we associate its $N \times N$ {\em
substitution matrix} $M=M_\om$, where $M(i,j)$ is the number of tiles of
type $A_i$ in the patch $\om(A_j)$. The substitution $\om$ is
called {\em primitive} if the substitution matrix is primitive,
that is, if there exists $k\in \NN$ such that $M^k$ has only
positive entries.}
\end{definition}
\begin{theorem}\cite{Gott} {\em (see also \cite[Sec.\,5]{Robi}).} \label{th-min}
\begin{enumerate}
\item[(i)] An FPC tiling system is repetitive if and only if it is
minimal, that is, every orbit $\{\Sk-\bg:\ \bg\in \R^d\}$
is dense in $X$.
\item[(ii)]
An FPC substitution tiling system is minimal if and
only if the substitution is primitive.
\end{enumerate}
\end{theorem}
So far, much of the theory has focused on primitive tile
substitutions $\om$. In this paper we investigate what happens in
the absence of primitivity and repetitivity, in the context of
tiling systems which satisfy FPC.

\begin{lemma} \label{onto}
Let $\omega$ be an admissible substitution on the set of prototiles $\A$. Then $\omega:   X_{\A, \omega}\to X_{\A, \omega}$ is onto.
\end{lemma}
\begin{proof}
Let $\T\in X_{\A,\omega}$. For $n\geq 1$, we define
$$
P_n=[B_n(\b0)]^{\T}.
$$
Since the patches of $\T$ are legal, there exist $k_n\geq 1$,
$\bx_n\in \RR^d$ and $A^{(n)}\in \A$ such that
$$
P_n+\bx_n\subseteq \omega^{k_n}(A^{(n)}).
$$
Because $|\A|<\infty$, there exist an infinite subset $I\subseteq \NN$ and $A\in \A$ such that
$$
P_n+\bx_n\subseteq  \omega^{k_n}(A) \mbox{ for every } n\in I.
$$
Taking subsequences if it is necessary, we can suppose that
$k_n<k_{n+1}$. For $n\in I$, let $Q_n=\omega^{k_n-1}(A)$. Observe
that $P_n+\bx_n\subseteq \omega(Q_n)$. Let $\T_{A}\in \Xa$ be
such that $A\in \T_A$ (the tiling $\T_A$ exists by the definition
of admissible substitution), and   let
$\T_n=\omega^{k_n-1}(\T_A)$. We have
$$
Q_n=\omega^{k_n-1}(A)\subseteq \omega^{k_n-1}(\T_A)=\T_n,
$$
which implies $P_n+\bx_n\subseteq \omega(\T_n)$. By
compactness of $X_{\A,\om}$, there exist a subsequence
$(n_j)_{j\geq 0}$ and $\T'\in \Xa$ such that
$$
\lim_{j\to \infty} (\T_{n_j}-\varphi^{-1}(\bx_{n_j}))=\T'.
$$
Since for every $n\in I$,
$$
P_n \subseteq \omega(\T_n - \varphi^{-1} (\bx_n)),
$$
we get $\omega(\T')=\T.$
\end{proof}
\begin{lemma} \label{lemma.power} We have $X_{\A,\om} = X_{\A,\om^k}$ for $k\ge 2$.
\end{lemma}
\begin{proof}
The inclusion $\supseteq$ is clear. For the other inclusion, let
$P$ be a legal patch for $\om$. Then $P$ occurs in some
$\om^n(A),\ A\in \A$. By assumption, there is a tiling $\S$ in
$X_{\A,\om}$ which contains a tile of type $A$. By
Lemma~\ref{onto}, there exists a tiling $\S'\in X_{\A,\om}$
such that $\om^{n(k-1)} (\S') = \S$. Then a tile of type $A$ is in
some $\om^{n(k-1)}(A')$ for some $A'\in \A$, hence $P$ occurs in
$\om^{nk}(A')$ which implies that $P$ is legal for $\om^k$.
\end{proof}
\subsection{Substitution matrix.}
Let $M\in \M_{N\times N}(\ZZ_+)$. The graph $G(M)$ associated to
$M$ is the directed graph whose set of vertices is
$\{1,\ldots,N\}$, such that there is an edge from $i$ to $j$ if
and only if $M(i,j)>0$. An equivalence relation is defined on the
set of vertices of $G(M)$ as follows: $i\thicksim j$ if and only
if $j=i$ or there is a path in $G(M)$ from $i$ to $j$ as well as a
path from $j$ to $i$. The matrix is {\em irreducible} if and only
if all the vertices of $G(M)$ are equivalent, otherwise, it is
{\em reducible}. We call the equivalence classes of $G(M)$ the {\em irreducible components}, see \cite[4.4]{LM}.\\

\indent
We say that an equivalence class $\alpha$ {\it has access} to the
equivalence class $\beta$, or that $\beta$ is accessible from
$\alpha$, if and only if $\alpha=\beta$ or if there exists a path
in $G(M)$ from a vertex in $\alpha$ to a vertex in $\beta$. This relation is denoted by $\alpha \succeq \beta$.
In a similar way we say that the vertex $i$ has access to $\beta$ if
there is a path in $G(M)$ from $i$ to a vertex in $\beta$.\\

\indent
For an equivalence class $\alpha$ we denote by $M_{\alpha}$ the
irreducible submatrix (diagonal block) of $M$ corresponding to the restriction of
$M$ to $\alpha$.\\

\indent
Now we return to our tiling substitution $\om$ and let $M=M_\om$ be the substitution matrix. We identify the vertex set of the graph
$G(M)$ with the prototile set $\Ak$. Since the tiling dynamical system is
completely determined by the space $\Xa$, we can, in view of Lemma~\ref{lemma.power}, replace $\om$ by $\om^k$ for $k\ge 2$.
The substitution matrix for $\om^k$ is clearly $M^k$. By raising a reducible non-negative matrix to a power we can get rid
of the ``cyclic structure'' of the irreducible components (this will increase the number of irreducible components if there is
nontrivial cyclic structure), see \cite[4.5]{LM} for details. Thus, we can (and will) assume, without loss of generality, that
\begin{equation} \label{eq-irre}
\mbox{\em every irreducible block of $M$ is either primitive or equals $[0]$.}
\end{equation}
\subsection{Minimal and maximal components.}
Consider the irreducible components of the graph $G(M)$ which are
{\em maximal} in the partial order $\succeq$. In other words, a
component $\alpha$ is maximal if it is not accessible from any
other component. Denote by $\A_1,\ldots,\A_m$ the subsets of the
prototile set $\A$ corresponding to these maximal components.
Observe that $m\ge 1$. By the definition of the substitution
matrix, we have $\om(\A_i) \seq \A_i^+$, so that
$\om_i:=\om|_{\A_i}$ is a tile substitution on the prototile set
$\A_i$. By assumption (\ref{eq-irre}), this substitution is
primitive, so $(X_{\A_i,\om_i},\R^d)$ is a minimal dynamical system
by Theorem~\ref{th-min}. Let $X_i = X_{\A_i,\om_i}$ for $i \le m$.
In the next lemma we show that these are precisely the minimal
components of the tiling dynamical system. (It seems
counter-intuitive that minimal components correspond to maximal
irreducible components of the graph, but this is just the
consequence of definitions. One can also consider the graph
$G(M^T)$ for the transpose of the substitution matrix; this
reverses the direction of edges, so the minimal components of the
dynamical system correspond to minimal irreducible components of
$G(M^T)$.)
\begin{lemma} \label{lem-minim} Suppose that $\om$ is an admissible FPC tile substitution satisfying (\ref{eq-irre}). Then
\begin{enumerate}
\item[(i)]
the minimal components of the tiling dynamical system $(\Xa,\R^d)$ are $(X_i,\R^d)$, for $i=1,\ldots,m$; they satisfy $\om(X_i) \seq X_i$;
\item[(ii)]
for any tiling $\Tk\in \Xa$ and a prototile $A\in \A_j$
which occurs in $ \Tk$, the orbit closure $\clos\{\Tk-\bg:\
\bg\in \R^d\}$ contains every minimal component $X_i$ such
that $\A_j$ is accessible from $\A_i$.
\end{enumerate}
\end{lemma}
\begin{proof} We already know that $(X_i,\R^d)$ is minimal and $\om(X_i) \seq X_i$. Part (ii) will imply that there are no other minimal components; thus, it remains
to prove (ii). Let $A\in \A_j$ and $A+\bx \in \Tk$ for some
$\bx\in \R^d$ and $j\geq 1$. By Lemma \ref{onto}, there exists a
sequence $k_n\uparrow \infty$ and prototiles $A^{(n)}$ such that
$$
A+\bx\in \om^{k_n}(A^{(n)})+\bx_n \seq \Tk\ \ \mbox{for
some}\ \bx_n \in \R^d.
$$
Passing to a subsequence, we can assume that $A^{(n)} = B$ for all
$n$. Let $1\leq i\leq m$ be such that $\A_j$ is accessible from
$\A_i$, then $B$ is accessible from $\A_i$, hence $\om^s(B)$
contains a prototile of type $\A_i$ for some $s\in \NN$. Since
$\T$ contains a translates of $\omega^{n+s}(B)$ for arbitrarily
large $n$, the closure of the
 orbit  of $\Tk$ contains
$X_{\A,\om_i}=X_i$.
\end{proof}
\begin{corollary}
\label{components}
There are at  most $|\A|$ minimal components in $X_{\A,\omega}$.
\end{corollary}
Suppose that the matrix $M$, and the graph $G(M)$, have $\ell$
irreducible components. It is also useful to consider the {\em
minimal} irreducible components of the graph $G(M)$ (or maximal
irreducible components of $G(M^T)$). The corresponding subsets
$\A_j$ of the prototile set are characterized by the property that
for any $i\ne j$ and $A\in \A_i$, the substitution $\om(A)$ does
not contain tiles of type $\A_j$. Suppose there are $p$ such
components; the corresponding prototile sets are
$\A_{\ell-p+1},\ldots,\A_\ell$. Note that if $j\ge \ell-p+1$ and
$A\in \A_j$, then the substitution $\om(A)$ necessarily contains a
tile of type $\A_j$, since otherwise the substitution is not
admissible. Thus, the corresponding matrices $M_j$ are nonzero and
hence primitive by (\ref{eq-irre}). Let
$$
Y_j:= \{\Tk \in \Xa:\ \Tk\ \mbox{contains a tile of type}\ \A_j\}.
$$
We call $Y_j,\, j={\ell-p+1},\ldots,\ell$, the {\em maximal components} of the tiling space $\Xa$.
\begin{lemma} \label{lem-maxim} Suppose that $\om$ is an admissible FPC tile substitution satisfying (\ref{eq-irre}). Then
\begin{enumerate}
\item[(i)] The subsets $Y_j,\, j={\ell-p+1},\ldots,\ell$, are mutually disjoint,
open in $\Xa$, and invariant under the translation $\R^d$-action and the
substitution $\Z_+$-action;
\item[(ii)]
for any $\Tk \in Y_j,\, j={\ell-p+1},\ldots,\ell$, we
have $Y_j\seq\clos\{\Tk-\bg:\,\bg\in \R^d\}$;
\item[(iii)] $\bigcup_{j=\ell-p+1}^\ell Y_j$ is dense in $\Xa$.
\end{enumerate}
\end{lemma}
\begin{proof} (i) This is immediate; $Y_j$ is invariant under $\om$ because the irreducible component $M_j$ is non-zero.\\
\indent
(ii) Exactly as in the proof of Lemma~\ref{onto}, we obtain $A\in
\A$, an infinite set $I_\Tk$ and $\bx_n \in \R^d$ for $n\in
I_\Tk$ such that
\begin{equation} \label{vspom}
[B_n(\b0)]^{\Tk} + \bx_n \seq \om^{k_n}(A)\ \ \mbox{ for}\  n\in
I_\Tk.
\end{equation}
The assumption $\Tk \in Y_j$
implies that $A\in \A_j$. But $M_j$ is primitive, so we can use the same $A\in \A_j$ for any $\Sk \in Y_j$. This implies that the orbit of $\Tk$ contains $\Sk$ in its
closure, as desired.\\
\indent
(iii) Let $\Tk \in \Xa$. Again we find $A$ and $I_\Tk$ satisfying
(\ref{vspom}). By the definition of the graph $G(M)$, the vertex
(prototile) $A$ has access to one of the maximal components
$\A_j$, with $\ell-p+1 \le j \le \ell$. This means $A+ \bz \in
\om^{k_0}(A')$ for some $A' \in \A_j$ and $k_0 \in \NN$. Let
$\Tk'\in \Xa$ be a tiling containing $A'$, which exists by
admissibility. Then $\om^{k_0+k_n}(\Tk')-\varphi^{k_n}(\bz) -
\bx_n \in Y_j$, and this sequence converges to $\Tk$.
\end{proof}
\subsection{Non-negative matrices.}
Let $M\in \M_{k\times k}(\ZZ_+)$ and let $\alpha$ be an
equivalence class of $G(M)$. Denote by $\rho_{\alpha}$ its
spectral radius. The class $\alpha$ is {\it distinguished} if
$\rho_{\alpha}>\rho_{\beta}$ for every class $\beta\neq \alpha$
which has access to $\alpha$. In particular, if $\alpha$ is not
accessible from any other class, then $\alpha$ is distinguished. A
real number $\lambda$ is called a {\it distinguished eigenvalue}
of $M$ if there exists a non-zero vector $\bx\geq \b0$ such that
$M\bx=\lambda \bx$. The following theorem extends the
Perron-Frobenius Theorem to reducible matrices, see \cite{S},
\cite{TS2} and \cite{Vic} for the proof.
\begin{theorem}
\label{Perron-Frobenius} Let $M\in \M_{k\times k}(\ZZ_+)$.
\begin{enumerate}
\item[(i)] A real number $\lambda$ is a distinguished eigenvalue if and
only if there exists a distinguished class $\alpha$ for $M$ such
that $\rho_{\alpha}=\lambda$.
\item[(ii)] If $\alpha$ is a distinguished class of $G(M)$, then there
exists a unique (up to scaling) non-negative eigenvector
$\bv_{\alpha}=(\bv_1,\cdots, \bv_k)$ corresponding to
$\rho_{\alpha}$ with the property that $\bv_i>0$ if and only
if the class $\alpha$ is accessible from the vertex $i$.
\end{enumerate}
\end{theorem}
We call $\bv_{\alpha}$ the {\it distinguished eigenvector} of
$M$ corresponding to $\alpha$.
\begin{definition}{Let $M\in \M_{k\times k}(\ZZ_+)$. The
{\it core} of $M$ is
$core(M)=\bigcap_{n\geq 1}M^n(\RR_+^k).$ }
\end{definition}
The next theorem can be found in \cite{TS1}. %it follows from Theorem~\ref{Perron-Frobenius} rather easily, after one shows that $core(M)$ is a simplicial cone \cite{pullman}.
\begin{theorem}
\label{generated-core} Let $M$ be a non-negative integer square
matrix verifying (\ref{eq-irre}). Then  the core of $M$ is the
cone generated by the distinguished eigenvectors of $M$.
\end{theorem}
\subsection{Core of the substitution matrix.}
Let $\A$ be a finite set of prototiles and  let $\omega:\A\to
\A^+$ be a substitution with expansion map $\varphi:\RR^d\to
\RR^d$ and substitution matrix $M_{\omega}=M\in\M_{\A\times
\A}(\ZZ_+)$. (Here we abuse notation a little and label the rows
and columns of $M$ by the elements of the prototile set $\A$.)
Suppose that $\alpha_1,\cdots,\alpha_l$ are the equivalence
classes (irreducible components) for the matrix $M$. For every
$1\leq i\leq l$ we denote by $M_i$ the restriction of $M$ to the
class $\alpha_i$. Let $X_1,\cdots, X_m$ be the minimal components
of $X_{\A,\omega}$, and let $\A_i\subseteq \A$ be the minimal
subset such that $X_i\subseteq X_{\A_i}$. As already mentioned, we
may assume without loss of generality that $X_i=X_{\A_i,\omega}$
and $\omega|_{\A_i}:\A_i\to \A_i^{+}$ is primitive, for every
$1\leq i\leq m$. Observe that for every $1\leq i\leq m$, the set
of prototiles $\A_i$ is equal to an equivalence class $\alpha_i$
of $M$. The restriction $M_i$ of $M$ to $\A_i$ is the matrix
associated to the substitution $\omega|_{\A_i}$. The equivalence
classes of the matrix $M$ can be ordered in such a way that $M$
has the block upper-triangular form
$$
M= \left(
  \begin{array}{ccccccc}
      M_1 & 0 & \cdots & 0 & \ast & \cdots & \ast \\
          0 & M_2 & \cdots & 0 & \ast  & \cdots & \ast \\
          \vdots & \vdots & \ddots & \vdots & \vdots & \cdots& \vdots \\
              0 & 0 & \cdots & M_m & * & \cdots & * \\
              0 & 0 & \cdots & 0 & M_{m+1} & \cdots & \ast \\
                  \vdots & \vdots & \cdots & \vdots & \vdots & \ddots & \vdots \\
                  0 & 0 & \cdots & 0 & 0 & \cdots & M_l \\
                    \end{array}
                \right)
$$
where $\ast$ stand for arbitrary non-negative integer matrices. Observe that for each  $m+1 \le i \le l$ there must be at least one non-zero off-diagonal
block, because those $M_i$ correspond to non-minimal prototile subsets.
% Taking powers of $M$ we can assume that $M$ satisfies (\ref{eq-irre}).

The number $\lambda_0=|\det(\varphi)|$ is a left eigenvalue of $M$
with $\bv_0=(\vol(A))_{A\in \A}\in\RR_+^{\A}$ as an associated
eigenvector. Indeed,  the $A$-coordinate of $\bv_0^TM$ is
equal to $\vol(\omega(A))$ and $\vol(\omega(A))=\lambda_0\vol(A)$,
for every $A\in\A$ by (\ref{def-sub}).
The next proposition shows that $\lam_0$ is the unique distinguished eigenvalue of $M_{\omega}$ and characterizes the core of $M_{\omega}$.
\begin{proposition}
\label{core} Let $M\in \M_{\A\times\A}(\ZZ_+)$ be the substitution
matrix associated to the substitution $\omega:\A\to\A^+$ with
expansion map $\varphi:\RR^d\to\RR^d$. Assume that (\ref{eq-irre}) holds. Let $X_1, \cdots, X_m$ be
the minimal components of $X_{\A,\omega}$, and let $\A_i$ be the corresponding prototile sets, so that $X_i = X_{\A_i}$. Then
\begin{enumerate}
\item[{(i)}] $\lam_0 = |\det(\varphi)|$ is the unique distinguished eigenvalue of $M$.
\item[{(ii)}]
$\A_1,\cdots, \A_m$ are all the different distinguished
classes of $M$, and $\rho(M_i)=|\det(\varphi)|$ for every $1\leq
i\leq m$.
\item[{(iii)}]
Let $1\leq i\leq m$. If $\bv_{i}$ is the distinguished eigenvector
of $M$ corresponding to the class $\A_i$, then
$\bv_i(p)>0$ if and only if $p\in \A_i$.
\item[{(iv)}] $core(M)$ is the cone generated by the vectors
$\bv_1,\cdots, \bv_m$.
\end{enumerate}
\end{proposition}
\begin{proof}
We know that $M_i$, for $i=1,\ldots,m$,
is the substitution matrix for the primitive tile substitution $\om|_{\A_i}$, hence $\rho(M_i)$ are distinguished eigenvalues of $M$. However,
$\rho(M_i) = \rho(M_i^T) = |\det(\varphi)|=\lam_0$ for $i=1,\ldots,m$, by the definition of tile substitution (\ref{def-sub}).
On the other hand, the fact that $M^T$ has a strictly positive eigenvector implies, by \cite[Theorem 13.4.6]{gant} that $\rho(M_i) = \rho(M_i^T)< \rho(M^T) =\rho(M)$
for $i=m+1,\ldots,l$.
(This is also easy to see directly: every non-minimal class $i$ ``loses entropy'' under the substitution, in view of $\om(\A_i) \not \seq \A_i$.) Therefore, the classes
$i=m+1,\ldots,l$ are not distinguished by definition. Now all the claims of the proposition immediately follow from Theorem~\ref{Perron-Frobenius}
and Theorem~\ref{generated-core}.
\end{proof}
\subsection{Invariant measures and transverse measures.}\label{transversal}
Let $(X,\RR^d)$ be a tiling dynamical system (it need not be a substitution system, although for our purposes we can take $X = \Xa$).
An {\it invariant} measure of
$(X,\RR^d)$ is a measure $\mu:\B(X)\to \overline{\RR}_+$ such that
$\mu(U)=\mu(U-\bv)$, for every $U\in \B(X)$ and
$\bv\in\RR^d$.\\
\indent Recall that we consider every prototile $A\in\A$ centered
at $\b0$. The {\it center of the tile} $t\in\T\in X$ is the point
$\bx_t\in\RR^d$ for which there exists $A\in\A$ such that
$t=A+\bx_t$. Let $\eta>0$ be such that the interior of
the support of $A$ contains the closure of $B_\eta(\b0)$, for every $A\in\A$.\\
\indent
By the {\it transversal} of $X$ we mean the set $\Gamma\subseteq
X$ of all the tilings $\T$ in $X$ for which there exists a
prototile $A\in \A$ such that $A\in \T$. In other words, if
$L_{\T}\subseteq\RR^d$ is the set of all the points $\bx\in
\RR^d$ for which there exist $\T\in X$ and $A\in\A$ such that
$A+\bx$ is a tile of $\T$,  then
$$
\Gamma=\{ \T\in X: \b0\in L_{\T} \}.
$$
\indent We say that a patch $P$ is {\it $X$-admissible} if there
exists $\T\in X$ such that $P\subseteq \T$. We denote by
$\Lambda_X$ the collection of all the $X$-admissible patches $P$
for which there exists a prototile $A\in\A$ such that $A\in P$. In
other words, this is the collection of patches such that $\b0\in
\Int(\supp(P))$ and $\b0$ coincides with the center of some tile
$t\in P$. Every $X$-admissible patch is a translate of some patch
in $\Lambda_X$. For every $P\in\Lambda_X$ we  define
$$
C_{P}=\{\T\in X: P\subseteq \T \}.
$$
The sets $C_{P}$ are subsets of $\Gamma$.\\
\indent
Equipped with the induced topology, the space $\Gamma$  is compact
and totally disconnected, with a countable base of clopen sets (the
collection of the sets $C_P$ is a base of the topology). The
collection of Borel sets $\B(\Gamma)$ of $\Gamma$ is equal to
$\B(X)\cap\Gamma$.
\begin{remark}
{\rm If $U\in \B(\Gamma)$ and $\Theta\subseteq \RR^d$ is an open
set, then $U+\Theta\in \B(X)$. To verify that, observe that this is
true for the sets $C_P$, with $P\in \Lambda_X$. Verifying that the
collection $\M=\{U\in \B(\Gamma): U+\Theta\in \B(X)\}$ is a
$\sig$-algebra, we get the result. }\end{remark}
A {\it transverse measure } on $\B(\Gamma)$   is a measure $\nu:
\B(\Gamma)\to \overline{\RR}_+$ such that $\nu(A)=\nu(A-\bv)$,
for every  $A\subseteq \B(\Gamma)$ and $\bv\in\RR^d$ for which
$A-\bv\subseteq \Gamma$.\\
\indent
Tiling spaces have been studied as {\em laminations}, or {\em translation surfaces}, see \cite{BG,BBG}.
Our definition agrees with the notion of transverse measure for laminations.
There is a one-to-one correspondence
between finite invariant measures and finite transverse measures (see \cite[Section 5]{BBG}), however, in all the existing literature it is assumed that the tiling system is minimal. We need this correspondence in the non-minimal case, and also for $\sig$-finite positive measures.
Therefore, we include details about this in the Appendix (Section \ref{Apendix A}) for the reader's convenience.
If $\mu$ is an invariant measure, we denote by $\mu^T$ the associated transverse measure.
%The correspondence is given by Lemma~\ref{invariant-transversal} in the Appendix.
%%%%%%%%%%%%%%%%%%%%%%%%%%%%%%%%%%%%%%%%%%%%%%%%%%%%%%%%%%%%%%%%%%%%%%%%%%%%%%%%%%%%%%%%%%%%%%%%%%%%%%%%%%%%%%%
\section{Finite invariant measures.}\label{finmeas}
\begin{theorem} \label{th-finite} Let $(X_{\A,\om},\R^d)$ be a self-affine tiling dynamical system. Then
all finite invariant measures are supported on the minimal components.
\end{theorem}
\noindent {\em Scheme of the proof.}
Let $\mu$ be a finite invariant measure. We can assume that $\mu$ is ergodic. It is easy to see that the sets
$C_P$, where $P$ is an admissible patch,
generate the Borel $\sig$-algebra on $\Gam$. Therefore, the Borel $\sig$-algebra on $\Xa$ is generated by the sets $C_P+\Theta$, with  $\Theta\seq B_r(\b0)$, and their translates. The claim will follow once we show
that $\mu(C_P+\Theta)=0$ for every patch $P$ which does not occur in $\bigcup_{i=1}^m X_i$, the union of
minimal components. To this end, we will use the pointwise ergodic theorem and show that the frequency of such a patch $P$ in any tiling $\T\in X_{\A,\om}$ equals zero.
We have two cases to consider: (a) $P$ contains a tile from $\Ak':= \Ak \setminus
\bigcup_{i=1}^m \Ak_i$; (b) $P$ contains tiles from two distinct minimal components.
The following example shows why case (b) is needed.
\begin{example}\label{ex3}{\em
All the tiles have the unit square as its support and are distinguished only by the labels. Let $\Ak = \left\{\begin{tabular}{|c|} \hline 0 \\ \hline \end{tabular}\,,
\begin{tabular}{|c|} \hline 1 \\ \hline \end{tabular}\,,\begin{tabular}{|c|} \hline 2 \\ \hline \end{tabular}\right\}$; the substitution $\om$ is given by
$$
\begin{tabular}{|c|} \hline 0 \\ \hline \end{tabular}\ \to\  \begin{tabular}{|c|c|c|c|} \hline 0 & 0 & 0 & 0 \\ \hline 0 & 0 & 1 & 0 \\ \hline 0 & 1 & 2 & 0 \\ \hline 0 & 0 & 0 & 0
\\ \hline \end{tabular}  \,, \ \ \ \ \ \ \begin{tabular}{|c|} \hline 1 \\ \hline \end{tabular}\  \to\ {\begin{tabular}{|c|} \hline 1 \\ \hline \end{tabular}}^{\,\,4\times 4}\,,
\ \ \ \ \ \ \begin{tabular}{|c|} \hline 2 \\ \hline \end{tabular}\  \to\ \begin{tabular}{|c|} \hline 2 \\ \hline \end{tabular}^{\,\,4\times 4}.
$$
where ${\begin{tabular}{|c|} \hline a \\ \hline \end{tabular}}^{\,\, k\times k}$ stands for the $k\times k$ square filled with identical prototiles labelled $a$.
We have two minimal components, corresponding to the prototiles labelled $1$ and $2$. The tiling space $X_{\A,\om}$ contains tilings which have a half-plane filled by $1$'s and another
half-plane filled by $2$'s. As we show, such tilings have zero measure for any finite invariant measure $\mu$.
}
\end{example}

\medskip

Now we turn to the details.
We use the pointwise ergodic theorem for $\R^d$-actions, which was first proved by Wiener \cite{wiener}, with averaging over balls centered at the origin. For us, another
averaging sequence is more convenient. Let
\be \label{eq-vanh}
F_n:= \varphi^n(B_1(\b0)).
\ee
It is well-known that the pointwise ergodic theorem holds with $F_n$ used instead of the balls (see e.g.\ \cite{tempel,chatard}).
\begin{theorem} \label{thm-pet}
(Pointwise Ergodic Theorem for $\R^d$-actions) Let $\mu$ be an ergodic invariant probability measure for the system $(\Xa,\R^d)$. Then for any $f\in L^1(X_{\A,\om},\mu)$,
\be \label{th-pet}
\frac{1}{\vol(F_n)} \int_{F_n} f(\T-\bg)\,d\bg \to \int f(\S)\,d\mu(\S),\ \
\mbox{as}\ n\to \infty,
\ee
for $\mu$-a.e.\ $\T \in X_{\A,\om}$.
\end{theorem}

For a bounded set $F\seq \R^d$ and $r>0$ let
$$
F^{+r} := \{\bx\in \R^d:\ \dist(\bx,F) \le r\}.
$$
\begin{definition} \label{def-vanHove}
{A van Hove (F{\o}lner) sequence for $\R^d$ is a sequence $\{\Theta_n\}_{n\ge 1}$ of bounded measurable subsets of $\R^d$ satisfying
\begin{equation} \label{eq-vanHove}
\lim_{n\to \infty} \frac{\vol((\partial \Theta_n)^{+r})}{\vol(\Theta_n)} = 0,\ \ \mbox{for all}\ r>0.
\end{equation}
}
\end{definition}
It is easy to see that our sequence $\{F_n\}_{n\ge 1}$ defined in (\ref{eq-vanh}), is van Hove. Moreover, $\{\varphi^n(\supp(t))\}_{n\ge 1}$ is van Hove for any prototile $t\in \A_i,\ i\le m$.
Indeed, the latter follows from the fact that $\vol(\partial(\supp(t)))=0$ for a prototile in a minimal component, which was proved for primitive tile substitutions in \cite[Prop.1.1]{prag}.
\\
\indent
Given a set $F\seq \R^d$, patch $P$, and a tiling $\T$, denote by $N(F,P,\T)$ the number of patches in $\T$ equivalent to $P$ such that $\supp(P)\seq F$.
Note that for any set $F$,
\be \label{eq-puga}
N(F,P,\T) \le \delta^{-1}\vol(F),
\ee
where $\delta$ is the volume of a ball contained in the interior of every prototile.
For $\T\in
X_{\A,\omega}$ and a patch $P$, the {\it frequency} of $P$ in
$\T$ with respect to $F_n$ (which is our default averaging sequence) is defined by
$$
\freq_{\T}(P)=\lim_{n\to\infty}\frac{N(F_n,P,\T)}{\vol(F_n)}\,,
$$
whenever the limit exists.\\
\indent
For $f$ the indicator function of the set $C_P+\Theta$ and $\beta = \diam(P)+\diam(\Theta)$ we obtain
\begin{eqnarray*}
0\le \frac{1}{\vol(F_n)}\int_{F_n} f(\T - \bt)\,d\bt & \le & \frac{N((F_n)^{+\beta},P,\T)}{\vol(F_n)} \cdot \vol(\Theta) \\
& \le & \left( \frac{N(F_n,P,\T)}{\vol(F_n)} + \delta^{-1} \frac{\vol((\partial F_n)^{+\beta})}{\vol(F_n)} \right)\cdot \vol(\Theta)
\end{eqnarray*}
hence $\freq_{\T}(P)=0$ for $\mu$-a.e.\ $\T$ will imply $\mu(C_P+\Theta)=0$ by (\ref{th-pet}), in view of $\{F_n\}$ being van Hove.
%Recall that $\A' = \A \setminus \bigcup_{i=1}^m \A_i$, i.e.\ $\A'$ is the set of prototiles from ``non-minimal'' components.
\begin{lemma} \label{lem-asymp}
Let $M$ be the substitution matrix for a tile substitution with expansion map $\varphi$. Then
for every $A\in \A'$ and $B\in \A$,
$$
\lim_{n\to \infty} \frac{M^n(A,B)}{|\det(\varphi)|^n} = 0.
$$
%\end{enumerate}
\end{lemma}
\begin{proof}
This follows from the structure of the matrix $M$ and Proposition~\ref{core}, using that all classes in $\A'$ are non-distinguished.
A direct reference is \cite[Theorem 9.4]{S}.
\end{proof}
\begin{lemma} \label{lem-fre1} Let $A$ be a tile in $\A'$ and $\T\in \Xa$. Then $\freq_\T(A)=0$.
\end{lemma}
\begin{proof}
Fix $n\in \NN$. Recall that $\om:\ \Xa\to \Xa$ is onto, hence there exists $\T'\in \Xa$ such
that $\om^n(\T') = \T$. Observe that
$$
F_n= \varphi^n(B_1(\b0)) \seq \supp (\om^n([B_1(\b0)]^{\T'})).
$$
Let $[B_1(\b0)]^{\T'} = \{t_1+\bx_1,\ldots,t_k+\bx_k\}$ where $t_i$ are prototiles (not necessarily distinct) and $\bx_i\in \R^d$. Note that $k\le K$ where $K$ is a uniform constant
by the FPC property.
We have
$$
\frac{N(F_n,A,\T)}{\vol(F_n)} \le \frac{\sum_{i=1}^k N(\om^n(t_i + \bx_i),A,\T)}{\vol(F_n)} =  \frac{\sum_{i=1}^k M^n(A,t_i)}{|\det(\varphi)|^n \,\vol(B_1(\b0))}\,,
$$
and the claim follows from Lemma~\ref{lem-asymp}.
\end{proof}
\begin{lemma} \label{lem-fre2} Let $P$ be an $\Xa$-admissible patch such that all its prototiles belong to $\A\setminus \A'$ but $P$ is not admissible for any of the minimal components
$X_i$. Then $\freq_\T(P)=0$ for any $\T\in \Xa$.
\end{lemma}
\begin{proof}
Let $n=s+\ell$ and find $\T'\in \Xa$ such that $\om^n(\T') = \T$. Consider
$$C:= \om^s([B_1(\b0)]^{\T'})=\{t_{1,s}+\by_{1,s},\ldots,t_{k_s,s} + \by_{k_s,s}\},$$
where $t_{i,s}$ are prototiles (not necessarily distinct) and $\by_{i,s}\in \R^d$.
Then $F_n \seq \supp (\om^\ell(C))$ and we have
$$
N(F_n,P,\T) \le N(\om^{\ell}(C),P,\T) = N(\bigcup_{i=1}^{k_s} \om^\ell(t_{i,s} + \by_{i,s}), P, \T).
$$
Observe that if $P$ has nonempty intersection with $\om^\ell(t_{i,s}+\by_{i,s})$, then either $t_{i,s} \in \A'$ or $\supp(P)$ intersects the boundary
$\partial(\supp(\om^\ell(t_{i,s}+\by_{i,s})))$ for $t_{i,s} \not\in\A'$. Thus, in view of (\ref{eq-puga}),
\be \label{eq-nuga}
N(F_n,P,\T) \le \delta^{-1} \sum_{i\le k_s:\ t_{i,s} \in \A'} \vol(\supp (\om^\ell(t_{i,s}))) + \delta^{-1}
\sum_{i\le k_s:\ t_{i,s} \not\in \A'} \vol((\partial(\supp (\om^\ell(t_{i,s})))^{+\alpha}))
\ee
where $\alpha = \diam(P)$.
Fix $\ep>0$. By Lemma~\ref{lem-fre1}, there exists $s_0$ such that for $s\ge s_0$ we have
$$
\#\{i\le k_s:\ t_{i,s} \in \A'\} \le \ep |\det\varphi|^s.
$$
Denoting by $V_{\max}$ the maximal volume of a prototile we obtain
\begin{equation} \label{eq-new1}
\sum_{i:\ t_{i,s} \in \A'} \vol(\supp (\om^\ell(t_{i,s})))  \le  \ep|\det\varphi|^s \max_{i,s} \vol(\supp(\om^\ell(t_{i,s})))
\le \ep|\det\varphi|^{s+\ell} V_{\max}.
\end{equation}
On the other hand, $\{\supp(\om^\ell(t))\}_{\ell\ge 1}$ is a van Hove sequence for any prototile $t\not\in \A'$, hence there exists $\ell_0$ such that for any $t\not\in \A'$, for $\ell \ge \ell_0$,
$$
\vol((\partial(\supp (\om^\ell(t))))^{+\alpha}) \le \ep \,\vol(\supp (\om^\ell(t))).
$$
Combining this with (\ref{eq-nuga}), (\ref{eq-new1}), and using that
$$
\sum_{i=1}^{k_s} \vol(\om^{\ell}(t_{i,s})) = \vol(\om^{s+\ell}([B_1(\b0)]^{\T'}) \le |\det\varphi|^{s+\ell} KV_{\max}
$$
yields
$$
N(F_{s+\ell},P,\T) \le \ep \delta^{-1} |\det\varphi|^{s+\ell} (1+K) V_{\max}\ \ \ \mbox{for}\ s\ge s_0, \ell\ge \ell_0,
$$
and the claim follows.
\end{proof}

\medskip

Now Theorem~\ref{th-finite} is proved by the scheme given after its statement. We also obtain the following
\begin{theorem}
\label{pi1} There is an affine bijection between the set of finite
invariant measures of $X_{\A,\omega}$ and $core(M)$. The finite
ergodic measures of $X_{\A,\omega}$ are in one-to-one
correspondence with the distinguished eigenvectors of $M$.
\end{theorem}
\begin{proof}
Let $\mu$ be a finite ergodic measure. Theorem
\ref{th-finite} implies that $\mu$ is supported on a
minimal component $X_i$. Since $X_i$ is a minimal substitution
system, \cite[Theorem 3.1]{So1}), \cite[theorem 3.3]{So1},
\cite[Corollary 3.5]{So1} and \cite[8.2.11]{HJ} imply that $\mu$
is determined by the Perron eigenvalue of the matrix $M_i$, where
$M_i$ is the restriction of $M_{\omega}$ to the minimal component
$X_i$. Indeed,  $(\mu^T(C_A))_{A\in\A_i}$ is a Perron eigenvector
of $M_i$. Since $\mu^T(C_D)=0$ for every $D\in\A'$, Theorem
\ref{Perron-Frobenius}  implies that $(\mu^T(C_A))_{A\in\A}$ is a
distinguished eigenvector of $M$ associated to the class $\A_i$.
Since $core(M)$ is the cone generated by its distinguished
eigenvectors, we get that $(\mu^T(C_A))_{A\in \A}$ is in $core(M)$.
It is straightforward to show that the function $\mu\mapsto
(\mu^T(C_A))_{A\in\A}$, defined on the set of finite invariant
measures of $(X_{\A,\omega},\R^d)$ to $core(M)$,  is affine and
bijective.
\end{proof}

%%%%%%%%%%%%%%%%%%%%%%%%%%%%%%%%%%%%%%%%%%%%%%%%%%%%%%%%%%%%%%%%%%%%%%%%%%%%%%%%%%%%%%%%%%%%%%%%%%%%
 \section{Recognizability} \label{recognizability}
A substitution $\om$ is {\em non-periodic} if the dynamical system
$(\Xa,\R^d)$ has no periodic points, that is, if $\Tk\in \Xa$ and
$\Tk+\bv=\Tk$ for $\bv\in \RR^d$, then $\bv=\b0$.
 In this section we show the following theorem:
 \begin{theorem}
 \label{th-recog}
 Let $\omega:\A\to \A^+$ be an admissible tiling substitution. The function $\omega: X_{\A,\omega}\to X_{\A,\omega}$ is one-to-one if and only if $\omega$ is non-periodic.
 \end{theorem}
 It is straightforward to show that a periodic substitution is not one-to-one. Indeed, if $\T\in X_{\A,\omega}$ is such that $\T=\T+\bv$ for some $\bv\in \RR^d\setminus\{\b0\}$,
 then Proposition \ref{onto} implies that for every $i\geq 1$ there exists $\T_i\in X_{\A,\omega}$ such that $\omega^i(\T_i)=\omega^i(\T_i+\varphi^{-i}(\bv))=\T$. Observe that
 for $i\geq 1$ sufficiently large,  $\T_i\neq \T_i+\varphi^{-i}(\bv)$ since $\varphi^{-i}(\bv)$ is close to zero.
This proves that $\omega$ is not one-to-one.\\ \indent
Theorem \ref{th-recog} was already proved for primitive substitutions in \cite{So}, so here we focus on the non-primitive case.

 \medskip

We also obtain ``partial recognizability'' for a class of
substitutions with periodic points which we now define. We can
assume, without loss of generality, that (\ref{eq-irre}) holds,
and $X_1,\ldots, X_m$ are the minimal components of
$X_{\A,\omega}$.  We say that $X_j$ is {\em periodic} if there
exists $\Tk\in X_j$ and $\bv\ne \b0$ such that $\Tk+\bv =
\Tk$. Then $\bv$ is a translational period for all tilings in
$X_j$. Denote by $\Aper$ the set of prototiles which occur in
periodic minimal components and $\Anonp:= \A \setminus \Aper$. For
any legal patch $P$ let $P|_{\rm  nonp}$ be the subpatch of all
$\Anonp$ tiles in $P$.
\begin{definition} \label{def-border}
We say that a substitution $\om$ satisfies the {\em non-periodic border condition} if
\begin{equation} \label{eq-bord1}
\forall\,\A\in \Anonp,\ \partial (\supp(\om(A)))\seq \supp(\om(A)|_{\rm nonp}).
\end{equation}
\end{definition}
\begin{definition} \label{part-recog}
A tile substitution $\om$ is said to be {\em partially recognizable} if for every $\T\in \Xa$ which contains a tile of
$\Anonp$ type there is a unique tiling $\T'\in \Xa$ such that $\om(\Tk') = \Tk$.
\end{definition}
\begin{theorem} \label{th-recog2}
An admissible tile substitution $\om$ satisfying the non-periodic border condition is partially recognizable.
\end{theorem}
\begin{corollary} \label{cor-period}
If $\om$ has non-periodic border and $\T\in \Xa$ contains a tile in $\Anonp$, then $\T$ is non-periodic.
\end{corollary}
The non-periodic border condition is not necessary for the claim of Theorem~\ref{th-recog2} to hold (see Example~\ref{ex-saf} below), but
the following example shows that some assumption on the substitution $\om$ is needed.
It is plausible that, without any additional assumptions, a {\em non-periodic} tiling $\Tk$ has a unique preimage under $\om$; however, this remains an open question.
\begin{example} {\em
Let $\Ak = \left\{\begin{tabular}{|c|} \hline 0 \\ \hline \end{tabular}\,,
\begin{tabular}{|c|} \hline 1 \\ \hline \end{tabular}\right\}$;
$$
\begin{tabular}{|c|} \hline 0 \\ \hline \end{tabular}\ \to\  \begin{tabular}{|c|c|c|} \hline 0 & 0 & 0 \\ \hline
                                      0 & 0 & 0 \\ \hline
                                      0 & 0 & 0 \\ \hline
\end{tabular}\,
,\ \ \ \ \ \
\begin{tabular}{|c|} \hline 1 \\ \hline \end{tabular}\  \to\  \begin{tabular}{|c|c|c|} \hline 0 & 1 & 0 \\ \hline
                                                                                0 & 1 & 0 \\ \hline
                                                                0 & 1 & 0 \\ \hline
                                        \end{tabular}
$$
Note that all tilings in $\Xa$ are periodic under the vertical shift, but according to our definition, only the prototile labelled 0 (which is in the minimal component) is periodic.
Thus, the conclusion of Corollary~\ref{cor-period} is violated (of course, the non-periodic border condition does not hold).
}
\end{example}

\medskip

Now we start preparation for the proofs.
Clearly, a non-periodic substitution has non-periodic border, so Theorem~\ref{th-recog2} implies Theorem~\ref{th-recog}.
Our argument is based on the method of \cite{HRS} where a new proof of recognizability for primitive tile substitutions was given (it applied to
a more general class of tilings, not just translationally FPC).
We should note that a direct proof of Theorem~\ref{th-recog} is simpler, and we will indicate which parts can
be skipped if one is only interested in non-periodic substitutions.

\medskip

Recall that $\eta>0$ is such that the support of
every prototile in $\A$ contains the closed ball $\ov{B}_\eta(\b0)$ in its interior. Let
$\gamma=\max_{A\in \A}\{ \diam(A) \}$, and $1<\lambda_1\leq
\lambda_2<\infty$ are positive numbers such that
\begin{equation} \label{eq-norm}
\lambda_1\|\bx\|\leq \|\varphi(\bx)\|\leq
\lambda_2\|\bx\|, \mbox{ for every } \bx\in \RR^d.
\end{equation}
Since $\varphi$ is expansive, we can find a norm in $\RR^d$ with
this property, and balls in this section will always be considered with respect to this norm. Then for every $n\geq 1$ and $\by\in
\RR^d$,
\begin{equation} \label{eq-expand}
B_{\lambda_1^nr}(\varphi^n(\by))\subseteq
\varphi^n(B_r(\by))\subseteq
B_{\lambda_2^nr}(\varphi^n(\by))
\end{equation}
The next definition was introduced in \cite{HRS}.

\medskip

Let $W$ be an $X_{\A,\omega}$-admissible patch. For every $n\geq
0$, let $\P_n(W)$ be the set of
 $X_{\A,\omega}$-admissible patches $P$ satisfying:
  \begin{enumerate}
   \item
    $\omega^n(W)\subseteq \omega^n(P)$;
     \item
      $\omega^n(W)$ is not included in $\omega^n(P')$, for any proper subpatch $P'$ of $P$.
       \end{enumerate}
Actually, \cite{HRS} used only $\P_n(t)$ for a single tile $t$; this would be sufficient if we were to restrict ourselves to non-periodic tilings.
\begin{lemma}
 \label{lema1-p}
  Let $W$ be an $X_{\A,\omega}$-admissible patch.
   \begin{enumerate}
    \item[{\bf (i)}]
     Let $n\geq 0$. For every $P\in \P_n(W)$, we have
      $
       \supp(P)\subseteq B_{2\gam}(\supp(W)).
        $
     \item[{\bf (ii)}]
      $\{W\}=\P_0(W)\subseteq \P_1(W)\subseteq \P_2(W)\subseteq\cdots$.
       \end{enumerate}
        \end{lemma}

 \begin{proof}
  (i) Let $n\geq 0$ and  $P\in \P_n(W)$. Since $\omega^n(W)\subseteq \omega^n(P)$, we have that $\supp(W) \subseteq \supp(P)$. Let $P'$ be the set of
all tiles in $P$ whose supports intersect $\supp(W)$. Then $P'\subseteq P$,
  and we have $\om^n(W) \subseteq \om^n(P') \subseteq \om^n(P)$. By
  the definition of $\P_n(W)$ we must have $P'=P$, and the desired
  property follows, since $\gam$ is the maximal diameter of a
  $X_{\A,\om}$-tile.

(ii) It is clear that $\{W\}=\P_0(W)$.
 Let $n\geq 0$ and $P\in \P_n(W)$.  The definition of the set $\P_n(W)$ implies that $\omega^{n+1}(W)\subseteq \omega^{n+1}(P)$.
 If $P'$ is a subpatch of $P$ for which $\omega^{n+1}(W)\subseteq
 \omega^{n+1}(P')$, then the support of $\omega^n(W)$ is included
 in the support of $\omega^n(P')$. This implies that
 $\omega^n(W)\subseteq \omega^n(P')$, and from definition of
 $\P_n(W)$, we get $P'=P$. This shows that $\P_n(W)\subseteq
 \P_{n+1}(W)$.
  \end{proof}

 Let $\P(W)=\bigcup_{n\geq 0}\P_n(W)$. The FPC assumption and part (i) of  Lemma \ref{lema1-p} imply that $\P(W)$ is finite up to translation.
 In the non-periodic case one can show that $\P(t)$ is finite, for any tile $t$. In the non-periodic border case, this is no longer true, and we
 have to work with ``first coronas'' or ``collared tiles'' containing at least one non-periodic tile.
More precisely, consider all legal patches of the form $[\supp(t)]^\T,\ t\in \T$,
for some $\T\in X_{\A,\om}$;
there are finitely many of them, up to translation. We choose a
representative for each of the translation-equivalent classes, and
denote their collection by $\F$. Denote by $\Fnonp$ the set of
patches in $\F$ which contain a tile of type $\Anonp$.
%More precisely,
%\begin{equation} \label{eq-prim}
%\forall\,F\in \Fnonp\ \exists\,t\in F:\ t-x_t\in \Anonp\
%\mbox{and}\ B_r(0)\seq\supp(t).
%\end{equation}
Now we can state the key proposition needed in the proof of Theorem~\ref{th-recog2}.

\begin{proposition}
\label{prop-principal2}   There exists $M\in \NN$ such that for
any $\T\in X_{\A,\omega}$, $n\geq 0$ and $\bx,\by\in
\RR^d$, if $P=\om^n(Z)$, with $Z - \by \in \Fnonp$, such that
     $$
     P\subseteq \T, \,\,\,\, P+\bx \subseteq \T,
     $$
     then
     $$
     \bx\in  \varphi^{n-M}(B_\eta(\b0))\ \mbox{implies}\  \bx=\b0.
     $$
\end{proposition}
The proof will be based on several lemmas.

\begin{lemma}
 \label{lem-claim1}
  There exists $R_0>0$ such that for every $\T\in X_{\A,\omega}$ and $\bx\in \RR^d$, the ball $B_{R_0}(\bx)$ contains an $X_i$-admissible $\T$-patch,
  for some  $1\leq i\leq m$.
   \end{lemma}
   \begin{proof}
   Suppose  that  for every $R>0$ there exist $\T_R\in X_{\A,\omega}$
   and $\bx_R\in \RR^d$, such that the ball $B_{R}(\bx_R)$ does not
   contain patches belonging to any minimal component.  Compactness
   of $X_{\A,\omega}$ implies there exists a sequence $R_n
   \uparrow\infty$ such that $(\T_{R_n}-\bx_{R_n})_{n\geq 0}$ converges to
   some tiling $\T\in X_{\A,\omega}$. It follows that $\T$ does not
   contain patches from any minimal component, which is not possible
   because $\clos\{\T-\bg:\,\bg\in \R^d\}$ must contain at least one minimal component
   of $X_{\A,\omega}$.
\end{proof}
The strategy of the proof of Proposition~\ref{prop-principal2} is
to find a large sub-patch of $\om^n(t)\seq \om^n(Z)$, with
$t-\bx_t \in \Fnonp$, which belongs to a minimal component
$X_i$. This component may be non-periodic or periodic. The former
case is treated with the following two lemmas. The first one is a
special case of \cite[Lemma 3.2]{So}.
  \begin{lemma}\label{lema5}{\em (\cite[Lemma 3.2]{So})}
       Let $\omega:\A\to \A^+$ be a primitive non-periodic substitution, and let $\eta>0$ be such that the support of every prototile in $\A$ contains the ball $B_\eta(\b0)$.  Then there exists $N\in \NN$ such that, for any $\T\in X_{\A,\om}$, $l>0$, and $\bx,\by\in \RR^d$, if
            $$
             P\subseteq \T, \,\,\,\, P+\bx \in \T,   \,\,\,\, \varphi^l(B_\eta(\b0))+\by\subseteq  \supp(P),
              $$
                   then
                    $$
                     \bx\in  \varphi^{l-N}(B_\eta(\b0)) \mbox{ implies }   \bx=\b0.
                          $$
                       \end{lemma}
\begin{lemma}
\label{lem-period}   There exists $N\in \NN$ such that for any
$\T\in X_{\A,\omega}$, $n\geq 0$ and $\bx,\by\in \RR^d$,
if $P$ is an $X_i$-admissible patch, where $X_i$ is a non-periodic
minimal component, such that
     $$
          P\subseteq \T, \,\,\,\, P+\bx \subseteq \T,   \,\,\,\, \varphi^n(B_\eta(\b0))+\by\subseteq  \supp(P),
           $$
                then
                 $$
                  \bx\in  \varphi^{n-N}(B_\eta(\b0))\ \mbox{implies}\  \bx=\b0.
                       $$
                     \end{lemma}

\begin{proof}
Let $k\in \NN$ be such that $2B_\eta(\b0) = B_\eta(\b0)+B_\eta(\b0) \seq
\varphi^k(B_\eta(\b0))$. We can take $k =\lceil \log 2/\log \lam_1
\rceil$ by (\ref{eq-norm}). By Lemma \ref{lema5}, for $j\in
\{1,\cdots, m\}$, there exists $N_j\in \NN$ such that if  $Q$ is
an  $X_j$-admissible patch satisfying
$$
 Q\subseteq \T', \,\,\,\, Q+\bw \seq \T',   \,\,\,\, \varphi^n(B_\eta(\b0))+\bv\subseteq  \supp(Q),
 $$
  for some $\bw\in \RR^d\setminus\{\b0\}$,  $\bv\in \RR^d$ and $\T'\in X_j$, then
   $$
    \bw\notin \varphi^{n-N_j}(B_\eta(\b0)).
     $$
We claim that the desired statement holds for $N=k+\max\{N_j:  1\leq j\leq m \}$.
Let $\T\in X_{\A,\omega}$, $\bx\in \RR^d\setminus \{\b0\}$ and
$P$ be an  $X_i$-admissible patch such that
   $$
        P\subseteq \T, \,\,\,\, P+\bx \seq \T,   \,\,\,\, \varphi^n(B_\eta(\b0))+\by\subseteq  \supp(P),
         $$
         for some $n\geq 0$ and $\by\in \RR^d$.
         Further, suppose that $\bx\in \varphi^{n-N}(B_\eta(\b0))\setminus \{\b0\}$. Clearly, $n>N$, since every tile contains a ball of radius $r$, so shifting a tile
         by a vector in $B_\eta(\b0)$ will result in a tile intersecting the original one, making $P, P+\bx\seq \T$ impossible.
Since $P$ is $X_i$-admissible, there exists $\T'\in X_i$ such that $P\seq \T'$. Consider
$$
P':= [\varphi^{n-k}(B_\eta(\b0))+\by]^{\Tk} \seq \Tk'.
$$
We want to apply Lemma~\ref{lem-period} to $P'$ and $\T'$. The
only thing we need to check is that $P'+\bx \seq \T'$. This
will follow if we show that $P'+\bx\seq P$, and to verify the
latter it suffices to check that
$$
\varphi^{n-k}(B_\eta(\b0))+\by+\bx \seq \varphi^n(B_\eta(\b0)) +
\by.
$$
However,
$$
\varphi^{n-k}(B_\eta(\b0))+ \bx \seq \varphi^{n-k}(B_\eta(\b0))+
\varphi^{n-N}(B_\eta(\b0)) \seq \varphi^{n-k}(B_\eta(\b0))+
\varphi^{n-k}(B_\eta(\b0)) \seq \varphi^n(B_\eta(\b0))
$$
by the definition of $k$. The proof is complete.
\end{proof}

\begin{proof}[Proof of Proposition~\ref{prop-principal2}]
Suppose that $\Tk, P$, and $\bx$ are as in the statement of
the proposition, with some $M\ge 1$, which will be determined below.
Suppose that $\bx\ne \b0$. Let $t$ be a tile of the patch $Z$ of
type $\Anonp$. Its support contains the ball
$B_\eta(\by)$ for some $\by\in \R^d$. Then
$$
\supp(P) = \supp(\om^n(Z))\supseteq \supp(\om^n(t))\supseteq
\varphi^n(B_\eta(\by)).
$$
Let $n_0\ge 0$ be the smallest integer satisfying $ \eta\lam_1^{n_0}
\ge R_0 $, that is,
$
n_0=\left \lceil \log_{\lambda_1}\left( \frac{R_0}{\eta}\right)\right
\rceil.
$
Here $R_0$ is the constant from Lemma~\ref{lem-claim1}.
Let $\Sk$ be any tiling in $X_{\A,\om}$ satisfying
$\om^{n-n_0}(\Sk) = \Tk$. Then $\varphi^{n_0}(B_\eta (\by))$
contains a ball of radius $R_0$, hence an $X_i$-admissible tile
$t'\in \Sk$, for some minimal component $X_i$, $1 \le i \le m$, by
Lemma~\ref{lem-claim1}. Therefore, $\varphi^n(B_\eta (\by))$,
and hence the patch $\om^n(t)\seq P$, contains an $X_i$-admissible
patch $P'=\om^{n-n_0}(t')\subseteq \Tk$. Now there are two cases.
If $X_i$ is non-periodic, then we apply Lemma~\ref{lema5} to
conclude that $\bx=\b0$, provided $M \ge N + n_0$ (where $N$ is
from Lemma~\ref{lem-period}). {\bf This concludes the proof in the
case when the substitution $\om$ is non-periodic.}

\medskip

It remains to treat the case when $X_i$ is periodic. The idea is
the following: since $P,P+\bx\seq \Tk$, we have that $P'+
\bx\seq P$ as long as $\supp(P'+ \bx)\seq \supp(P)$. Then
$P'+ \bx$ is also $X_i$-admissible. We can continue in this
manner as long as the translates of $P'$ by a multiple of $\bx$
remain in $\supp(P)$, and this works for individual tiles as well,
not necessarily for the entire $P'$. If $\bx$ is small
relative to the size of $P'$, we will obtain an entire ``tube''
from $P'$ to the border of $\om^n(t)$ which is $X_i$-admissible.
But this will lead to a contradiction with the non-periodic border
assumption. Now for the details. A slight complication arises because of the possibility that the interior of a tile is disconnected, so we actually take the ``connected component" of $P'$.

Let us continue with the proof of the proposition. We can assume that
$M\ge n_0$. Clearly, the assumptions imply $n\ge M$ (since
$\bx\in \varphi^{n-M}(B_\eta(\b0))$ is a non-zero translation
between two tiles in $\Tk$, and every prototile contains $B_\eta(\b0)$
in the interior of its support). Recall that
$P'=\varphi^{n-n_0}(t')$ is $X_i$-admissible, where $X_i$ is a
periodic minimal component, $P'\seq \om^n(t)\seq P$, and $t$ is a
tile of type $\Anonp.$ It follows by induction from
(\ref{eq-bord1}) that
\begin{equation} \label{eq-bord2}
\partial (\supp(\om^n(t)))\seq \supp(\om^n(t)|_{\rm nonp}).
\end{equation}
The tile $t'$ contains a ball $B_\eta(\bz)$ for some $\bz\in
\R^d$, hence $\varphi^{n-n_0}(B_\eta(\bz))\seq\supp(P')$. Consider
$$
V:= \mbox{\ the component of \ \ $\Int(\supp(\om^{n-n_0}(t')))$ \ \ containing $\varphi^{n-n_0}(B_\eta(\bz))$}.
$$
Clearly $[V]^\Tk\seq P'$, so all its tiles are of type $\Anonp$. Note that
$V\cap (V+\bx)\ne \es$ because
$$
\varphi^{n-n_0}(B_\eta(\bz))\cap
(\varphi^{n-n_0}(B_\eta(\bz))+\bx)\ne \es\ \ \mbox{for}\
\bx\in \varphi^{n-M}(B_\eta(\b0))\seq \varphi^{n-n_0}(B_\eta(\b0)).
$$
Let $k\ge 0$ be the largest integer such that $V+k\bx\seq\supp(\om^n(t))$. Then $[V]^\Tk+(k+1)\bx\seq \Tk$,
because $\om^n(t)\seq P$ and $P+\bx\seq \Tk$. Moreover,
$$
V+(k+1)\bx\not\seq \supp(\om^n(t))\ \ \mbox{and}\ \  (V+k\bx)\cap
(V+(k+1)\bx)\ne \es.
$$
It follows that $V+(k+1)\bx$ contains a point $\bz_1$ in the interior of a tile in $\Tk\setminus \om^n(t)$ and
$V + k\bx$ contains a point $\bz_2$ in the interior of a tile in $\om^n(t)$. The set
$$W:=(V+k\bx)\cup(V+(k+1)\bx)$$
is open and connected, being a union of two open connected sets with non-empty intersection. Thus it is arcwise connected. An arc in $W$ with the endpoints $\bz_1$ and $\bz_2$ must intersect the boundary of $\supp(\om^n(t))$.
The point of intersection belongs to a tile,
say, $t''$, of type $\Anonp$ by (\ref{eq-bord2}).
This tile is in $[W]^\Tk$, but the types of all tiles in $[W]^\Tk$ are those of the tiles in $[V]^\Tk$, hence they are $X_i$-admissible.
Thus, $t''$ is of type
$\Aper$, which is a contradiction.
\end{proof}
\begin{lemma} \label{lem-finite-p}
There exists $N\in \NN$ such that for any patch $Z$ and
$\by\in \R^d$ such that $Z-\by\in \Fnonp$, we have
$\P_{n+1}(Z) = \P_n(Z)$ for all $n\ge N$.
\end{lemma}
\begin{proof}
Let $n\ge 0$ and $P \in \P_n(Z)$. Suppose that $\bx\in
\R^d\setminus \{\b0\}$ is such that $P+\bx \in \P_n(Z)$. We have
$$
\om^{n}(Z) \subseteq \om^n(P)\ \ \ \mbox{and}\ \ \ \om^n(Z) -
\varphi^n(\bx) \subseteq \om^n(P).
$$
We conclude from Lemma~\ref{lem-period} that $\varphi^n(\bx)
\not\in \varphi^{n-M}(B_\eta(\b0))$, hence $\bx\not\in
\varphi^{-M}(B_\eta(\b0))$, which implies
$$
\|\bx\| \ge \eta \lam_2^{-M}.
$$
From part (i) of Lemma~\ref{lema1-p} it follows that the supports
of translated copies of $P$ which belong to $\P_n(Z)$, are
contained in $B_{2\gam}(\supp(Z))$. Recall that $Z-\by\in
\Fnonp$ is a ``collared tile'' hence $\diam(Z) \le 3\gam$. Thus,
there are at most
$$
\frac{\vol(B_{5\gam+\eta \lam_2^{-M}/2}(\b0))}{\vol(B_{\eta \lam_2^{-M}/2})(\b0)}=:N_1
$$
copies of the patch $P$ in $\P_n(Z)$. The FPC ensures that there are finitely many patches in $\Fnonp$, up to translation.
Also by FPC, there are
at most $C$ patches, up to translation, whose support is contained in $B_{2\gam}(\supp(Z))$ for some $Z\in \Fnonp$. From this we deduce that
$\P_n(Z)$ has at most $CN_1$ elements. Since this is valid for every $n\ge 0$, from part (ii) of Lemma~\ref{lema1-p} it follows that $|\P(Z)| \le
CN_1.$
\end{proof}
%\medskip
We continue with the scheme of \cite{HRS}.
\begin{lemma} \label{lem-easy1}
Suppose $\P_n(W) = \P_{n+1}(W)$, where $W$ is a legal patch. If
$\S\in X_{\A,\om}$ is such that $\om^{n+1}(W) \subseteq \om(\S)$,
then $\om^n(W) \seq \S$.
\end{lemma}
\begin{proof}
Let $\S'\in X_{\A,\om}$  be such that $\om^n(\S') = \S$. Then
$\om^{n+1}(W)\seq \om(\S) = \om^{n+1}(\S')$, hence there exists $P
\in \P_{n+1}(W)$ such that $P\seq \S'$. Since $P\in \P_n(W)$ we
have $\om^n(W) \seq \om^n(P) \seq \om^n(\S') = \S$.
\end{proof}
%\medskip
{\em Proof of Theorem~\ref{th-recog2}.} Let $\T_1 \in X_{\A,\om}$
be such that $\om(\T_1) =\T$, and further, suppose $\T_n \in
X_{\A,\om}$ are such that $\om(\T_n) = \T_{n-1}$ for $n\ge 2$.
Let $t_n \in \T_n$ be such that $\supp(t_n) \ni \b0$, and let $Z_n =
[\supp(t_n)]^{\T_n}$. Then $\om^n(Z_{n+1}) \seq \T_1$ and
$$
\bigcup_{n\ge 1} \supp(\om^n(Z_{n+1})) = \R^d,
$$
hence $Z_k$ are of type $\Fnonp$ for $k$ sufficiently large
(otherwise, $\T_1$ and $\T = \om(\T_1)$ contain only tiles from $\Aper$ contradicting our assumption),
that is, there exists $k_0$ such that $Z_k - \bz_k
\in \Fnonp$ for $k\ge k_0$.
We want to show that $\T_1$, with $\om(\T_1) = \T$, is uniquely
determined. To this end, consider any $\T'$, with $\om(\T') = \T$.
We have for $n\ge \max\{k_0,N\}$, by Lemma~\ref{lem-finite-p},
that $\P_{n+1}(Z_{n+1})= \P_n(Z_{n+1})$, hence by
Lemma~\ref{lem-easy1},
$$
\om^{n+1}(Z_{n+1}) \seq \T=\om(\T')\ \Longrightarrow\
\om^n(Z_{n+1}) \seq \T'.
$$
Therefore, $\T'$ contains the patches $\om^n(Z_{n+1})\seq \T_1$
for all $n$ sufficiently large, and these patches exhaust the
entire tiling. Thus, $\T'=\T_1$, as desired. \qed
%%%%%%%%%%%%%%%%%%%%%%%%%%%%%%%%%%%%%%%%%%%%%%%%%%%%%%%%%%%%%%%%%%%%%%%%%%%%%%%%%%%%%%%%%%%%%%%%
\section{Infinite invariant measures.}
\subsection{Non-negative matrices revisited}
We use the notation and results from Sections 2.5 and 2.6.
\begin{definition}{ Let $M\in \M_{k\times k}(\ZZ_+)$.  The {\em infinite core}
of $M$ is the set of all the vectors in
$core_{\infty}(M)=\bigcap_{n\geq 1}M^n(\overline{\RR}_+^k)$ where $\ov{\RR}_+ := \RR_+\cup \{\infty\}$.}
\end{definition}
We saw in Section 3 that $core(M)$ is isomorphic to the set of finite invariant measures. Here we will show that $core_{\infty}(M)$ is closely
related to the
set of ``nice'' invariant $\sig$-finite measures, under some mild assumptions. The goal of this subsection is to describe the infinite core.

\begin{lemma}
\label{help-lemma} Let $M_1$ and $M_2$ be two non-negative square
matrices of dimensions $n_1$ and $n_2$ respectively, such that
$M_1$ is primitive and $M_2$ has a positive eigenvalue $\rho_2>0$
associated to a positive eigenvector $\bv_2$. Let $C\neq 0$ be
a non-negative $n_1\times n_2$-dimensional matrix  and let
$\rho_1$ be the Perron eigenvalue of $M_1$. If there exists a
vector $\bx\in \overline{\RR}_+^{n_1}$  such that
$$
\left(\begin{array}{c}
  \bx \\
  \bv_2 \\
\end{array}\right)\mbox{ is in the infinite core of } M= \left(%
\begin{array}{cc}
  M_1 & C \\
  0 & M_2 \\
\end{array}%
\right),
$$
then $\bx=\infty$ whenever $\rho_1\geq \rho_2$.
\end{lemma}
\begin{proof}
Let $n>0$, $C_n$ and $\bx_n\in \overline{\RR}^{n_1}_+$ be such
that
$$
\left(%
\begin{array}{cc}
  M_1 & C \\
  0 & M_2 \\
\end{array}%
\right)^n = \left(\begin{array}{cc}
  M_1^n & C_n \\
  0 & M_2^n \\
\end{array}%
\right )
$$
and
$$
\left(\begin{array}{c}
  \bx \\
  \bv_2 \\
\end{array}\right)=M^n\left(\begin{array}{c}
      \bx_n \\
  \frac{\bv_2}{\rho_2^n} \\
\end{array}\right).
$$
Using the symbol ``$\ge$'' to denote the natural (component-wise)
partial order on vector spaces, we have
\begin{equation}
\label{Section6-eq1}
\bx=M_1^{n}\bx_n+C_n\frac{\bv_2}{\rho_2^n}\geq
C_n\frac{\bv_2}{\rho_2^n},
\end{equation}
and
\begin{eqnarray*}
C_{n+1}\frac{\bv_2}{\rho_2^{n+1}} & =
&\frac{1}{\rho_2^{n+1}}M_1^nC\bv_2+C_n\frac{\bv_2}{\rho_2^{n}}\\
  &=&
  \frac{1}{\rho_2}\left(\frac{\rho_1}{\rho_2}\right)^n\left(\frac{M_1}{\rho_1}\right)^nC\bv_2+
  C_n\frac{\bv_2}{\rho_2^n}\\
  &=&
  \frac{1}{\rho_2}\left[\sum_{k=0}^n\left(\frac{\rho_1}{\rho_2}\right)^k\left(\frac{M_1}{\rho_1}\right)^k\right]C\bv_2.
\end{eqnarray*}
Thus if $\rho_1\geq \rho_2$, we have
\begin{equation}
\label{Section6-eq2} C_{n+1}\frac{\bv_2}{\rho_2^{n+1}}\geq
\frac{1}{\rho_2}\left[\sum_{k=0}^n\left(\frac{M_1}{\rho_1}\right)^k\right]C\bv_2,
\end{equation}
which tends to $\infty$ with $n$, because $C\bv_2\neq \b0$ and
$\lim_{n\to \infty}(M_1/\rho_1)^n=\bw\bv>0$, where
$\bw$ and $\bv$ are   left and right  Perron eigenvectors
of $M_1$ respectively (see \cite[Theorem 8.5.1]{HJ}). Then from
equations (\ref{Section6-eq1}) and (\ref{Section6-eq2}) we
conclude that  $\bx=\infty$ when $\rho_1\geq \rho_2$.
\end{proof}

\medskip

%\begin{remark}{\rm  One direction of Theorem~\ref{Perron-Frobenius} (necessity
%of the class $\alpha$ being distinguished in order to have
%$\rho_\alpha$ as a distinguished eigenevalue) essentially follows
%from Lemma~\ref{help-lemma}. The other direction can be deduced
%from the observation that in the setting of
%Lemma~\ref{help-lemma}, if $\rho_1 < \rho_2$, then
%$$\left(\begin{array}{c}
%           \bx \\
%           \bv_2 \\
%       \end{array}\right),
%\ \ \ \mbox{with}\ \ \bx = \rho_2^{-1}(I-\rho_2^{-1}M_1)^{-1}
%C\bv_2 = \sum_{n=0}^\infty \rho_2^{-n-1} M_1^n C\bv_2$$ is
%a positive finite eigenvector for $M$.}
%\end{remark}
Suppose that $\alpha_1,\cdots,\alpha_l$ are the equivalence
classes associated to the matrix $M$. For
every $1\leq i\leq l$ we denote by  $M_i$ the restriction of $M$ to the
class $\alpha_i$. We assume that $M_i$ is primitive or equal to
$[0]$. When $M_i$ is primitive  we denote by $\rho_i$ the Perron
eigenvalue of $M_i$; if $M_i=[0]$ then $\rho_i=0$.
\begin{definition}
\label{definition1} {Let $M$ be a non-negative integer square
matrix with irreducible components $M_1,\cdots, M_l$. For every
$1\leq i\leq l$ such that $M_i$ is primitive we consider
\begin{enumerate}
\item $\J_i$---the set of indices $j\in\{1,\cdots,
l\}\setminus\{i\}$ such that the class $\alpha_j$ has access to a
class $\alpha_k$, where $\alpha_k$ has access to $\alpha_i$ and
$\rho_k\geq \rho_i$.
\item $\I_i$---the set containing $i$ and all the indices
$j\in\{1,\cdots,l\}\setminus \J_i$ such that the class $\alpha_j$
has access to the class $\alpha_i$ (then necessarily $\rho_j<\rho_i$).
\end{enumerate}}
\end{definition}
\begin{remark}
{\rm The classes $\alpha_j$, with $j\in\I_i$ do not have access to the
classes with indices in $\J_i$. The complement of $(\I_i\cup
\J_i)$ is the set of $j$ such that $\alpha_j$ does not
have access to $\alpha_i$.}
\end{remark}
Let $1\leq i\leq l$ be such that $M_i$ is primitive. The class
$\alpha_i$ is distinguished with respect to $M|_{\I_i}$, the
restriction of $M$ to the set of indices in $\I_i$. Then Theorem \ref{Perron-Frobenius}
implies that there exists a unique $|\I_i|$-dimensional normalized
positive vector $\bw_i$ such that
$M|_{\I_i}\bw_i=\rho_i\bw_i$. The restriction of
$\bw_i$ to $\alpha_i$ is an eigenvector of $M_i$ associated to
$\rho_i$.
\begin{definition}
\label{definition2} { For every $1\leq i\leq l$ such that $M_i$ is
primitive, we define $\by_i$ in $\overline{\RR}_+^k$ as
follows:
\begin{itemize}
\item The restriction of $\by_i$ to $\I_i$ is equal to
$\bw_i$.
\item The restriction of $\by_i$ to $\J_i$ is $\infty$ in every component.
\item The restriction of $\by_i$ to $(\I_i\cup
\J_i)^c$ is zero.
\end{itemize}
For every $1\leq i\leq l$, we define $\bz_i$ in
$\overline{\RR}_+^k$ as the vector whose restriction to
$\I_i\cup\J_i$ is infinite, and $\bz_i$ restricted to
$(\I_i\cup\J_i)^c$ is zero.
If $M_i=[0]$ we define $\by_i=\b0$.}
\end{definition}
%\begin{remark}
%{\rm If the class $\alpha_i$ is distinguished in $M$, then
%$\by_i$ is a distinguished eigenvector of $M$ associated to
%$\alpha_i$.}
%\end{remark}
 \begin{lemma}\label{infinite-distinguished}
Let $M$ be a non-negative integer $k\times k$ matrix with
irreducible components $M_1,\cdots, M_l$. For every $1\leq i\leq
l$  we have $\by_i\in core_\infty(M)$.
 \end{lemma}
\begin{proof}
If $M_i=[0]$ then $\by_i=\b0$ is in $core_\infty(M)$. Then we can
assume that $M_i$ is primitive.

For every $n\geq 1$, we define $\by_{n,i}\in \overline{\RR}_+^k$,
the vector such that
$$
\by_{n,i}|_{\I_i}=\frac{\bw_i}{\rho_i^n}, \hspace{4mm}
\by_{n,i}|_{\J_i}=\infty, \hspace{4mm}\mbox{ and }\hspace{4mm}
\by_{n,i}|_{(\I_i\cup\J_i)^c}=\b0.
$$
For $1\leq r\leq k$,  the $r$-coordinate of $M^n\by_{n,i}$ is equal
to
$$
 \sum_{s=1}^kM^n(r,s)\by_{n,i}(s)=\sum_{s\in\I_i}M^n(r,s)\by_{n,i}(s)+
\sum_{s\in\J_i}M^n(r,s)\by_{n,i}(s).
$$
Thus we have the following:\\
\indent
{\bf 1.} If $r\in (\I_i\cup \J_i)^c$, then $M^n(r,s)=0$ for every
$s\in \I_i\cup \J_i$. This implies the $r$-coordinate of
$M^n\by_{n,i}$  is equal to $\b0=\by_i(r)$.\\
\indent
{\bf 2.} If $r\in \I_i$, then $M^n(r,s)=0$ for every $s\in \J_i$.
This implies that the $r$-coordinate of $M^n\by_{n,i}$ is equal to
$$\sum_{s\in\I_i}M^n(r,s)\by_{n,i}(s)=\bw_i(r)=\by_{i}(r).$$
\indent
{\bf 3.} If $r\in \J_i$, then for every $n\geq 1$ there exists
$s_n\in \J_i$ such that $M^n(r,s_n)>0$. This implies that
$r$-coordinate of $M^n\by_{n,i}$ is  equal to
$M^n(r,s_n)\by_{n,i}(s_n)=\infty=\by_i(r)$.\\
\indent
The statements {\bf 1}, {\bf 2}, and {\bf 3} imply that
$\by_i=M^n\by_{n,i}$, for every $n\geq 1$, which shows
that $\by_i$ is in the infinite core of $M$.
\end{proof}
\begin{lemma}
\label{relation-between-eigenvalues} Let $M$ be a non-negative
integer $k\times k$ matrix with primitive irreducible components
$M_1,\cdots, M_l$. Then every  vector
$\bx \in core_{\infty}(M)$ can be written as
$$
\bx=\sum_{j=1}^l\lambda_j\by_j+
\sum_{j=1}^l\delta_j\bz_j,
$$
where $\lambda_1,\cdots, \lambda_l\geq 0$ and
$\delta_1,\cdots,\delta_l\in\{0,1\}$. Conversely, every vector
written in this way is in $core_{\infty}(M)$.
\end{lemma}
\begin{proof}
Let $\bx\in core_{\infty}(M)$. Theorem \ref{generated-core}
implies that when $\bx$ is finite, this vector is in the cone
generated by the vectors $\by_i$, for every $1\leq i\leq l$ such
that $\alpha_i$ is distinguished for $M$. If $\bx$ has all its
coordinates equal to $\infty$ or $zero$, it can be written as
$\sum_{i=1}^l\delta_i\bz_i$, for some $\delta_1,\cdots,
\delta_l\in\{0,1\}$. Thus we can assume  that $\bx$ has a finite
positive coordinate and an infinite coordinate. For $1\leq j\leq
l$, let $\bx_j$ be the $|\alpha_j|$-dimensional vector given by
the restriction of $\bx$ to $\alpha_j$. If some coordinate of
$\bx_j$ is positive and finite, then all the coordinates of
$\bx_j$ are positive and finite. Let $1\leq i\leq l$ be such that
$$i=\max\{1\leq k \leq l: 0<\bx_k<\infty\}.$$  If there exists
$i<k\leq l$ such that $\bx_{k}=\infty$ then for every $1\leq j
\leq l$ such that $\alpha_j$ has access to $\alpha_k$ we have
$\bx_j=\infty$. Then we can write
$$
\bx=\bu +\sum_{j=1}^l\delta_j\bz_j,
$$
where $\bu_j=\bx_j$ for every $1\leq j \leq i$,
$\bu_j=0$ for every $i<j\leq l$, and $\delta_1,\cdots,
\delta_l\in \{0,1\}.$ After rearranging the coordinates of
$\bx$ if it is necessary, we can suppose that there exists $1\leq
s\leq i$ such that $0\leq \bx_j<\infty$ for every $s\leq j\leq
i$, and $\bx_j=\infty$ for every $1\leq l<s$. The vector
$\bx|_{\alpha_s,\cdots,\alpha_i}=(\bx_s,\cdots,
\bx_i)$ is in the core of the restriction $M'$ of $M$ to the
classes $\alpha_s,\cdots, \alpha_i$. Then by Theorem
\ref{generated-core}, $\bx|_{\alpha_s,\cdots,\alpha_i}$ is in the cone generated by
the distinguished eigenvectors of $M'$. Observe that $\bv$ is
a distinguished eigenvector of $M'$ if and only if $\bv$ is
the restriction to $\alpha_s,\cdots,\alpha_i$ of a scalar multiple of the
vector $\by_j$, for some $s\leq j\leq i$ such that the class
$\alpha_j$ is distinguished in $M'$.
 Thus, using Lemma \ref{help-lemma}, we can write
$$
\bx=\sum_{j=1}^l\lambda_j\by_j+\delta_j\bz_j,
$$
where $\lambda_1,\cdots,\lambda_l\geq 0$ (with $\lambda_i>0$), and
$\delta_1,\cdots,\delta_l\in\{0,1\}$. The converse holds by
Lemma \ref{infinite-distinguished}.
\end{proof}
\subsection{Clopen nested partitions of the transversal.}
As in the previous sections, we consider a substitution $\omega$
defined on a set of prototiles $\A\subseteq \RR^d$. We denote by
$M\in\M_{\A\times\A}(\ZZ_+)$ the substitution matrix of $\omega$,
and $\A_1,\cdots, \A_l$ the equivalence classes associated to $M$.
We denote by $M_i$ the restriction of $M$ to the indices in
$\A_i$, and we assume that $M_i$ is primitive or equal to $[0]$.
We suppose there are equivalence classes which are not associated
to minimal components, namely, $\A_{m+1},\cdots,\A_l$, for some
$1\leq m< l$. We denote $\A'=\A\setminus
(\A_1\cup\cdots\cup\A_m)$.

\medskip

Recall that we denote by $\Gamma$ the transversal of $X_{\A,\omega}$,
and for every $A\in \A$, we set
$$ C_A=\{\T\in \Gamma: A\in\A\}.$$
The collection $\P_0=\{C_A: A\in \A \}$ is a clopen partition of
$\Gamma$. For  $n\geq 1$ and $A,B\in \A$,  we define
$$
D_{n,A}=\supp(\omega^n(A)),$$ $$J_{A,B}^{(n)}=\{\bv\in
D_{n,A}: B+\bv\subseteq \omega^n(A) \},$$
$$J_{A}^{(n)}=\bigcup_{B\in \A}J_{A,B}^{(n)},
$$
and
$$
\P_n=\{\omega^n(C_A)-\bv:\, \bv\in J_A^{(n)}, A\in \A\}.
$$
(These $\P_n$ have nothing to do with $\P_n(W)$ from Section 4.)
{\bf For the rest of this section we assume that the admissible substitution $\om$ is partially recognizable, see Definition~\ref{part-recog}, and also
\begin{equation} \label{eq-added}
\mbox{For every prototile $A\in \Anonp$, the patch $\om(A)$ contains a tile of type $\Anonp$.}
\end{equation}
}
Note that the latter condition is satisfied if $M$ has no components $[0]$, or if the non-periodic border condition holds.
\begin{lemma}
\label{partition} For every $n\geq 0$, the collection $\P_n$ is a
covering of $\Gamma$. Furthermore,
\begin{enumerate}
\item[(i)] For each $A\in \A$ and $n\geq 1$,
\begin{equation}
\label{partition-1} C_A=\bigcup_{B\in \A}\bigcup_{\bv\in
J_{B,A}^{(n)}}(\omega^n(C_B)-\bv).
\end{equation}
\item[(ii)] If $A,B\in \A$, $n\geq 1$, $\bv\in J_A^{(n)}$ and
$\bu\in J_B^{(n)}$ are such that $$(\omega^n(C_A)-\bv)\cap
(\omega^n(C_B)-\bu)\neq \emptyset,$$ then $A=B$ and
$\bv=\bu$, or $A, B\in \A_{\per}$.
\item[(iii)] Let $\ell>n$, $A\in \A$ , $B\in \Anonp$,
$\bv\in J_A^{(\ell)}$ and $\bu\in J_B^{(n)}$.  We have
$$(\omega^\ell(C_A)-\bv)\cap (\omega^n(C_B)-\bu)\neq \emptyset,$$ if and
only if $\omega^\ell(C_A)-\bv\subseteq \omega^n(C_B)-\bu$
and $B+\varphi^{-n}(\bv-\bu)\in \omega^{\ell-n}(A)$.
\end{enumerate}
\end{lemma}
\begin{proof}
(i) Proposition \ref{onto} ensures that $\P_n$ is a covering of
$\Gamma$ and implies (\ref{partition-1}). \\
\indent
(ii)
By partial recognizability,
if $A\neq B$ or $\bv\neq \bu$ then the
tilings in $(\omega^n(C_A)-\bv)\cap (\omega^n(C_B)-\bu)$
contain only tiles from $\Aper$. Hence the patches $\omega^n(A)$
and $\omega^n(B)$ only contain tiles in $\A_{\per}$. Condition
(\ref{eq-added}) implies that $A$ and $B$ are in
$\A_{\per}$.\\
\indent
(iii)
Let $D_1,\cdots, D_k\in \A$ and $\bx_1,\cdots, \bx_k\in \RR^d$
be such that $\omega^{\ell-n}(A)$ is the disjoint union
$\bigcup_{i=1}^k(D_i+\bx_i)$. Then $\omega^\ell(A)$ is equal
to the disjoint union
$\bigcup_{i=1}^k(\omega^n(D_i)+\varphi^n(\bx_i))$. This
implies that there exists $1\leq i\leq k$ such that
$\bz=\bv-\varphi^{n}(\bx_i)\in J^{(n)}_{D_i}$ and
$$
\omega^\ell(C_A)-\bv=
\omega^\ell(C_A)-\varphi^n(\bx_i)-\bz\subseteq
\omega^n(C_{D_i})-\bz.
$$
Then by hypothesis we have $(\omega^n(C_B)-\bu)\cap
(\omega^n(C_{D_i})-\bz)\neq \emptyset$. Since $B\in
\Anonp$, from part (ii) it follows that $B=D_i$,
$\bz=\bu$, and then $\omega^\ell(C_A)-\bv\subseteq
\omega^n(C_B)-\bu$ and $B+\varphi^{-n}(\bv-\bu)\in
\omega^{\ell-n}(A)$. The other direction of the equivalence is
immediate.
\end{proof}
\begin{remark} \label{rem-inf}
{\rm 1. If the substitution is non-periodic, then $(\P_n)_{n\geq 0}$
is a nested sequence of clopen partitions of $X_{\A,\omega}$. In
the minimal non-periodic one-dimensional symbolic case, the
sequences $(\P_n)_{n\geq 0}$ correspond to the sequence of
Kakutani-Rohlin partitions for minimal substitution subshifts
given in \cite{DHS}. In general, it is only a covering, but we will see below that it becomes a partition if we intersect it with the set of non-periodic tilings.\\
\indent
2. It is useful to give an informal interpretation of the covering $(\P_n)_{n\ge 0}$. Given a tiling $\T\in \Gamma$, we have a sequence of tilings $(\T_n)_{n\ge 0}$, with
$\T_0= \T$, such that $\om(\T_n) = \T_{n-1}$ for $n\ge 1$. This defines a sequence of ``supertilings'' obtained by composing the tiles of $\T$, with
``supertiles'' that are translates of $\om^n(A)$ for $A\in \A$. This sequence is uniquely defined if $\T$ contains a non-periodic tile. We consider
the supertile of order $n$ whose support contains the origin (it is uniquely defined since $\T$ is in the transversal). This determines the element of
$\P_n$ to which $\T$ belongs.
}
\end{remark}

\subsection{Necessary conditions for transverse measures.}
\begin{definition}Let $\mu^T$ be a transverse measure on $\B(\Gamma)$. For every $A\in \A$ and for
every $n\geq 0$, we define
$$
\mu^T_{n,A}=\mu^T(\omega^n(C_A)-\bv),
$$
where $\bv$ is a vector in $J_A^{(n)}$. The number $\mu^T_{n,A}$
does not depend on $\bv$ because $\mu$ is transverse. We denote by
$\mu^T_{n}$ the vector $(\mu^T_{n,A})_{A\in \A}$ and by $\wt{\mu}^T_n$ the vector $(\mu^T_{n,A})_{A\in \A'}$.
\end{definition}
\begin{lemma}
\label{coreN} Let $\mu$ be a transverse measure on $\B(\Gamma)$.
Then for every $A\in \A'$ and $\ell>n\geq 0$,
$$
\mu^T_{n,A}=\sum_{B\in \A'}M^{\ell-n}(A,B)\mu^T_{\ell,B}.
$$
Thus,
$$
\wt{\mu}^T_n = (M')^{\ell-n} \wt{\mu}^T_\ell,
$$
where $M'$ is the restriction of $M$ to the set of indices in $\A'$.
\end{lemma}
\begin{proof}
Let $A\in \A'$ and $\bu\in J_A^{(n)}$. From (iii) of Lemma
\ref{partition} we have
$$
\omega^n(C_A)-\bu=\bigcup_{B\in\A}\bigcup_{\bv\in
I^{(n)}_B}(\omega^\ell(C_B)-\bv),
$$
where $I_B^{(n)}$ is the set of $\bv$ in $J_B^{(n)}$ such that
$A+\varphi^{-n}(\bv-\bu)\in \omega^{\ell-n}(B)$. Since the
minimal components are $\omega$-invariant  we can restrict the
outer union to $\Ak'$:
$$
\omega^n(C_A)-\bu=\bigcup_{B\in\A'}\bigcup_{\bv\in
I^{(n)}_B}(\omega^\ell(C_B)-\bv).
$$
Observe that $|I_B^{(n)}|=M^{\ell-n}(A,B)$. Thus from (ii) of
Lemma \ref{partition} we obtain the desired equality.
\end{proof}
\begin{remark} {\rm
If $\A_{per}=\emptyset$ then the same proof shows that
$$\mu^T_{n,A}=\sum_{B\in\A}M^{\ell-n}(A,B)\mu^T_{\ell,B},
$$
for every $A\in \A$ and $\ell> n\ge 0$. It follows that
$\mu^T_0=M^n\mu^T_n$ for every $n>0$,  hence this vector belongs to $core_\infty(M)$. Thus  Lemmas \ref{help-lemma},
\ref{relation-between-eigenvalues} and \ref{coreN} imply that if
$\mu^T_{0,A}>0$ for some $A\in \A'$, then the restriction of $\mu^T_0$
to every minimal component which has access to $A$ is infinity
(because the component of $A$ is not distinguished). In the next
lemma we show the same for the general case.}
\end{remark}

\begin{lemma}
\label{finite} Let $\mu^T$ be a transverse measure on $\B(\Gamma)$
and let $1\leq p\leq m$. If there is $A\in \A_p$  for which there
exist $D\in \A'$ and $n>0$
verifying $\mu^T(C_D)>0$ and $M^n(A,D)>0$, then $\mu^T(C_E)=\infty$,
for every $E\in \A_p$.
\end{lemma}
\begin{proof}
Let
$$
\Xnonp:= \{\T\in \Xa:\,\T\ \mbox{contains a tile of type}\ \Anonp\}.
$$
This set is open and invariant under the $\R^d$ translation action.
Using partial recognizability, we obtain, exactly as in the proof of Lemma~\ref{partition}, the following statements:

\begin{enumerate}
\item[(ii$'$)] If $A,B\in \A$, $n\geq 1$, $\bv\in J_A^{(n)}$ and
$\bu\in J_B^{(n)}$ are such that
$$(\omega^n(C_A)-\bv)\cap
(\omega^n(C_B)-\bu)\cap \Xnonp\neq \emptyset,$$
then $A=B$ and
$\bv=\bu$.
\item[(iii$'$)] Let $\ell>n$, $A,B\in \A$,
$\bv\in J_A^{(\ell)}$ and $\bu\in J_B^{(n)}$.  We have
$$(\omega^\ell(C_A)-\bv)\cap (\omega^n(C_B)-\bu)\cap \Xnonp \neq \emptyset,$$ if and
only if $\omega^\ell(C_A)-\bv\subseteq \omega^n(C_B)-\bu$
and $B+\varphi^{-n}(\bv-\bu)\in \omega^{\ell-n}(A)$.
\end{enumerate}
In other words, by intersecting $\P_n$ with $\Xnonp$ we recover the nested partition properties even in the presence of periodic minimal components.

Next we can argue as in Lemma~\ref{coreN} to deduce that the vector
$$
(\mu^T(\om^n(C_A) \cap \Xnonp))_{A\in \A}
$$
belongs to the infinite core of $M$ for every $n\ge 0$. Let $D\in \A_i\seq \A'$ be such that $\mu(C_D) >0$. Then
$D\in \Anonp$, so
$$
\mu^T(C_D\cap \Xnonp) = \mu^T(C_D) >0.
$$
The assumption of the lemma implies that the class $\A_p$ has access to $\A_i$, and since $\A_p$ corresponds to a
minimal component, we have $\rho_p> \rho_i$. Now it follows by Lemma~\ref{relation-between-eigenvalues}
that
$$
\mu^T(C_E)\ge \mu^T(C_E\cap \Xnonp) = \infty\ \ \ \mbox{for all}\ E\in \A_p.
$$
\end{proof}

\subsection{Constructing infinite transverse measures.}
In Section \ref{finmeas} we proved that finite invariant
measures of the substitution tiling system $(X_{\A,\omega},\RR^d)=(X,\RR^d)$
are supported on its minimal components. Therefore, if $\mu$ is a
finite invariant measure and $\mu^T$ is the associated transverse
measure, then $\mu^T(C_D)=0$ for every prototile $D\in\A'$. In
this section we characterize the infinite $\sig$-finite invariant
measures $\mu$ for which there exists  $D\in\A'$ such that
$0<\mu^T(C_D)<\infty$. It follows from Lemmas \ref{coreN} and \ref{finite} that the values of $\mu^T$ on the elements of $\P_n$ belong to the
infinite core of $M$ (at least, if we exclude periodic components). Lemma~\ref{relation-between-eigenvalues} suggests that those which
correspond to ergodic measures
should come from the vectors $\by_i$. We will show that this is indeed the case, under some mild assumptions.

\medskip

Recall that $\A_1,\cdots,\A_m$ are the equivalence classes of
$M$ associated to the minimal components, and
$\A'=\A\setminus\cup_{i=1}^m\A_i=\A_{m+1}\cup\cdots\cup\A_l$. Let
$m+1\leq i\leq l$ be such that $M_i$ is primitive. Let $\A_{\J_i^c}$
be the set of prototiles $A\in\A$ such that if $A\in \A_j$ then
$j\in \J_i^c$ (see Definition \ref{definition1} for $\J_i$).
Equivalently, $\AJC$ is the set of prototiles with indices in $\I_i$ and those which have no access to the class $\A_i$.
By definition, $\A_i$ is a distinguished class for the restriction of $M$ to the indices in $\AJC$. A class $\A_j$, with $j\le m$,
corresponding to a minimal component, is in $\AJC$ if and only if it has no access to $\A_i$.
We define
$$\Gamma_i=\bigcup_{n\geq
0}\bigcup_{A\in\AJC}\bigcup_{\ \ \bv\in
J_A^{(n)}}(\omega^n(C_A)-\bv),$$ and $\F_i$,  the collection of
subsets of $\Gamma_i$ given by
$$
\F_i=\{\omega^n(C_A)-\bv: A\in\AJC,\ \bv\in J_A^{(n)},
n\geq 0\}.
$$
Recalling the informal description of partition elements from Remark~\ref{rem-inf}, we note that the tilings in $\Gam_i$
are those for which the ``supertiles'' of order $n$ containing the origin are of ``type'' $\AJC$ for $n$ sufficiently large.
 Observe that if this is the case for some $n_0$, then this is also true for $n > n_0$, since if a tile of type $\A_j$ occurs in $\om^n(B)$ for $B\in
\A_k$, then $j$ has access to $k$, and so either $j=k$, or $k$ does not have access to $j$.
Let $\by_i\in\overline{\RR}^{\A}_+$ be the vector given in
Definition \ref{definition2} for the class $\A_i$. For  $n\geq 0$,
let $\by_{n,i}\in\overline{\RR}^{\A}_+$ be such that
$\by_i=M^n\by_{n,i}$
(that is, $\by_{n,i}=\by_{i}/\rho^n_i$, where $\rho_i$ is the
Perron eigenvalue of $M_i$). We define the function $\phi_i:
\F_i\cup\{\emptyset\}\to \RR_+$ by $\phi_i(\emptyset)=0$ and
\be\label{def-phi}
\phi_i(\omega^n(C_A)-\bv)= \by_{n,i}(A), \mbox{ for every
} A\in \AJC \mbox{ and } n\geq 0.
\ee
Since $\by_i|_{\J_i^c}<\infty$, this function is well-defined. A
standard argument shows that the function $\phi_i^*: 2^{\Gam_i}\to
\overline{\RR}_+$ given by
$$
\phi_i^*(U)=\inf\left\{\sum_{n\in\NN}\phi_i(C_n): (C_n)_{n\in\NN}\subseteq
\F_i\cup\{\emptyset\}, \ U\subseteq \bigcup_{n\in\NN}C_n\right\}
$$
is an outer measure. Observe that $\phi_i^*$ is well-defined
because $\F_i$ is countable and the union of all the sets in $\F_i$ is
equal to $\Gamma_i$.
The collection
$$
\eta_i^*=\{U\subseteq \Gamma_i:\ \forall E\subseteq \Gamma_i,\
\phi_i^*(E)\geq \phi_i^*(E\cap U)+\phi_i^*(E\setminus U)\}
$$
is a $\sig$-algebra and the restriction of $\phi_i^*$ to $\eta_i^*$
is a complete measure (every negligible set with respect to
$\phi_i^*$ is in $\eta^*$).
\begin{lemma}\label{lem-meas1}
$\F_i \subseteq \eta_i^*$.
\end{lemma}
\begin{proof}
Let $U=\omega^n(C_A)-\bv\in \F_i$, with $A\in \AJC$ and $\bv\in J_A^{(n)}$. We need to show
\be \label{eq-want1}
\phi_i^*(E)\geq \phi_i^*(E\cap U)+\phi_i^*(E\setminus U)
\ee
for $E\seq \Gam_i$.\\
\noindent
\underline{\em Step 1.} Suppose first that
$E=\omega^m(C_B)-\bw$ is also in $\F_i$. If $A$  (resp.\ $B$)  is
in a minimal component, then $\phi_i(U)=0$ (resp.\ $\phi_i(E)=0$).
This immediately implies (\ref{eq-want1}).
Thus we can assume that $A$ and $B$ are in
$\Anonp$.
The inequality (\ref{eq-want1}) is also clear if $U\cap E=\emptyset$ or $E\setminus
U=\emptyset$.

If $m\geq n$, then $U\cap E\neq \emptyset$ implies that
$E\subseteq U$ by Lemma~\ref{partition}(iii); then $E\setminus U=\emptyset$ and we are done.

If $m<n$ and $U\cap E\neq \emptyset$, then $U\cap E=U$. In this
case we have
$$E\setminus U=\bigcup_{D\in \AJC}\bigcup_{\footnotesize{ \begin{array}{c}
  \bu\in J_D^{(n)} \\
  \omega^n(C_D)-\bu\subseteq E \\
  \bu\neq\bv \mbox{ if }D=A
\end{array} }}(\omega^n(C_D)-\bu),$$
which implies that
\begin{eqnarray*}
\phi_i^*(E\setminus U) & \leq & \sum_{D\in
\AJC}\sum_{\footnotesize{ \begin{array}{c}
  \bu\in J_D^{(n)} \\
  \omega^n(C_D)-\bu\subseteq E \\
\bu\neq\bv \mbox{ if }D=A
\end{array} }}\phi_i^*(\omega^n(C_D)-\bu)\\
  &=& \sum_{D\in \AJC}M^{n-m}(B,D)\by_{n,i}(D)-\phi_i^*(U).
\end{eqnarray*}
Therefore,
\begin{eqnarray*}
\phi_i^*(E) =  \by_{m,i}(B) & = & \sum_{D\in
\AJC}M^{n-m}(B,D)\by_{n,i}(D) \\
& \ge & \phi_i^*(E\setminus U)+\phi_i^*(U) = \phi_i^*(E\setminus U)+\phi_i^*(U\cap E),
\end{eqnarray*}
as desired.

\medskip

\underline{\em Step 2.}
If $E$ is any subset of $\Gamma_i$, then we have two
cases:

 (a) if $\phi_i^*(E)=\infty$,  then (\ref{eq-want1}) is clear.

 (b) If $\phi_i^*(E)<\infty$, then given $\varepsilon>0$, there
exists $(U_n)_{n\in\NN}\subseteq \F_i$ such that
      $$\sum_{n\in\NN} \phi_i^*(U_n)\leq \phi^*_i(E)+\varepsilon.$$
 Using Step 1 and the fact that $\phi^*_i$ is an outer measure,  we get
\begin{eqnarray*}
\phi_i^*(E\cap U)+\phi^*_i(E\setminus U)& \leq &
\phi^*_i(\bigcup_{n\in\NN} U_n\cap U)+\phi^*_i(\bigcup_{n\in\NN} U_n \setminus U)\\
& \leq & \sum_{n\in\NN}\left( \phi^*_i(U_n\cap U)+ \phi^*_i(U_n\setminus U)\right)\\
& \leq & \sum_{n\in\NN}\phi^*_i(U_n)\\
& \leq & \phi_i^*(E)+\varepsilon,
\end{eqnarray*}
concluding the proof.
\end{proof}

\medskip

In the sequel we will prove that under some conditions, the Borel sets
of $\Gamma_i$ (with respect to the induced topology) are contained
in $\eta^*$.
We define
$$
Y_i=\{\T\in \Xa: \T \mbox{ has a tile  equivalent to some }
A\in \A_i\}$$ and
$$\wt{\Gam}_i = \bigcup_{m\ge 0}\bigcap_{n\geq m}\bigcup_{A\in\A_i}\,\bigcup_{\bv\in
J_A^{(n)}}(\omega^n(C_A)-\bv)).
$$
Recall (Section 2.4) that in the special case when $\A_i$ is a
maximal irreducible component of the graph $G(M^T)$, the set $Y_i$
is a maximal component of the tiling space; here we consider the
same kind of set for an arbitrary non-minimal component.
\begin{lemma}
\label{supported-on-i} We have

{\em (i)} $\phi_i^*(\Gamma_i\setminus \wt{\Gam}_i) = 0$;

{\em (ii)} $\phi_i^*(\Gam_i \setminus Y_i) = 0$.
\end{lemma}
\begin{proof} Note that (ii) follows from (i). Indeed, $\T\in \wt{\Gam}_i$ if and only if eventually all ``supertiles'' containing the
origin (see Remark~\ref{rem-inf}) have a type from $\A_i$. Since we assumed that $M_i \ne [0]$, any substitution of a tile in $\A_i$
must contain a tile in $\A_i$, so $\wt{\Gam}_i \seq Y_i$.\\
\indent
It remains to verify (i). For every $j\in \J_i^c$ and $m\geq 0$, we define
$$
\Gamma_{i,j,m}=\bigcap_{n\geq
m}\bigcup_{A\in\A_j}\bigcup_{\bv\in
J_A^{(n)}}(\omega^n(C_A)-\bv) \hspace{3mm} \mbox{ and
}\hspace{3mm} \Gamma_{i,j}=\bigcup_{m\geq 0}\Gamma_{i,j,m}.
$$
We have
\begin{equation} \label{eq-new2}
\Gamma_i\setminus \wt{\Gam}_i = \bigcup_{j\in \J_i^c\setminus \{i\}}\Gamma_{i,j}.
\end{equation}
Indeed, $\Gamma_{i,j}$ is the set of tilings for which the supertiles containing the origin are eventually of type from $\A_j$. Since we assumed that
all non-zero components of $M$ are primitive, the types of the supertiles containing the origin must stabilize into types from one of the components,
hence the claim (\ref{eq-new2}).
Let $j\in \J_i^c$, $j\ne i$. If $j$ has no access to $i$, then the
definition of $\phi_i^*$ implies that $\phi_i^*(\Gamma_{i,j})=0$.
If $M_j=[0]$ then $\Gamma_{i,j}=\emptyset$. Thus we can assume
that $j\in \I_i$ and $M_j$ is primitive. Let $A\in \A_j$, $m\geq
0$ and $\bv\in J_A^{(m)}$. For every $n\geq m$ we have
$$
\Gamma_{i,j,m}\cap(\omega^m(C_A)-\bv)\subseteq
\bigcup_{B\in\A_j}\bigcup_{\footnotesize{\begin{array}{c}
  \bw\in J_B^{(n)} \\
  \omega^n(C_B)-\bw\subseteq \omega^m(C_A)-\bv \\
\end{array}}}\!\!\!\!(\omega^n(C_B)-\bw),
$$
which implies that
\begin{eqnarray*}
\phi_i^*(\Gamma_{i,j,m}\cap(\omega^m(C_A)-\bv)) & \leq &
\sum_{B\in\A_j}M^{n-m}(A,B)\frac{\by_i(B)}{\rho_i^n}\\
  &=
  &\rho_j^{-m}\sum_{B\in\A_j}\frac{M_j^{n-m}(A,B)}{\rho_j^{n-m}}\left(\frac{\rho_j}{\rho_i}\right)^n\by_i(B)
\end{eqnarray*}
Since $\lim_{n\to \infty}(M_j^{n-m}(A,B)/\rho_j^{n-m})$ exists and
is finite, and since $\rho_j<\rho_i$, we get
$$\phi_i^*(\Gamma_{i,j,m}\cap(\omega^m(C_A)-\bv))=0.$$
This implies that $\phi_i^*(\Gamma_{i,j,m})=0$, and then
$\phi_i^*(\Gamma_i\setminus \wt{\Gamma}_i)=0$.
\end{proof}

\medskip

For $n\geq 0$ and $A\in \A$ we define
$$
I_{n,r,A}=\bigcup_{B\in \A}\{\bv\in J^{(n)}_{A,B}:
(\supp(B)+\bv)\cap (\partial D_{n,A})^{+r}\neq
\emptyset \}.
$$
In other words, $I_{n,r,A}$ is the set of $\bv$ such that $B+\bv$  occurs in $\om^n(A)$ for some $B$, within distance $r$ from the boundary.
It is straightforward to show that the   tilings $\T\in \Gamma_i$
for which there exists $\bv\in\RR^d$ such that $\T-\bv\in
\Gamma\setminus \Gamma_i$ are in
\be \label{def-C}
C=\bigcup_{r\in\NN}\bigcup_{k\in\NN}\bigcap_{n\geq k}\bigcup_{A\in\AJC}\bigcup_{\bv\in
I_{n,r,A}}(\omega^n(C_A)-\bv).
\ee
The set $C$ is a ``bad set'' for us: it is the set of tilings in $\Gam_i$ which ``belong to the border'' in the
following sense: the union of supertiles containing the origin, discussed in Remark~\ref{rem-inf}, is not the entire space $\R^d$.
Observe that
\be \label{eq-goodint}
\T \in \Gam_i \setminus C \ \Longrightarrow\ \bigcap_{n\in \NN} (\om^n(C_{A_n}) - \bv_n) = \{\T\},
\ee
where
$\omega^n(C_{A_n})-\bv_n\in\P_n$ is the set containing $\T$, for every $n\in\NN$. Note also that $C$ is translation-invariant.
The next lemma gives a sufficient condition
for $C$ to be negligible with respect to $\phi_i^*$.
\begin{lemma}
\label{hip} If there exist $D\in \A_i$ and $n>0$ such that a
translate of $D$ appears in the interior of $\omega^n(D)$, then
$C$ is negligible with respect to $\phi_i^*$.
\end{lemma}
\begin{proof}
In view of Lemma \ref{supported-on-i}, it is enough to show that for every $m\ge 0$,
$$
C_{i,m}:=C\cap  \bigcap_{n\geq m}\bigcup_{A\in\A_i}\bigcup_{\bv\in
J_A^{(n)}}(\omega^n(C_A)-\bv)
$$
is negligible with respect to $\phi_i^*$. It is clear that
$$
C_{i,m} = \bigcup_{r\in \NN} \bigcap_{n\ge m} \bigcup_{A\in \A_i} \bigcup_{\bv\in I_{n,r,A}} (\omega^n(C_A) - \bv).
$$
We have $C_{i,m} = \bigcup_{r\in \NN} \bigcap_{n\ge m} C_{i,n,r}$ where
$$
C_{i,n,r} = \bigcup_{A\in \A_i} \bigcup_{\bv \in I_{n,r,A}} (\om^n(C_A) - \bv).
$$
Fix $r\in \NN$. It is enough to show that $\phi^*(\bigcap_{n\ge m} C_{i,n,r}) = 0$ for $m$ sufficiently large, and we will do this by estimating
$\phi^*(C_{i,n,r} \cap C_{i,m,r})$ for $n > m$. For $A \in \A_i$ consider the decomposition of $\om^n(A)$ into supertiles of order $m$, which is
the inflated decomposition of $\om^{n-m}(A)$ into tiles. By the ``border'' of $\om^{n-m}(A)$ we mean the patch of tiles whose supports intersect
the boundary of $\supp(\om^{n-m}(A))$. Applying $\om^m$ to the border increases its width, hence we can choose $m$ sufficiently large, so that
any tile in $\om^n(A)$ within distance $r$ from $\partial D_{n,A} = \partial(\supp(\om^n(A)))$ belongs to a supertile $\om^m(B)$ in the
inflated border. Then we have
$$
\phi_i^*(C_{i,n,r} \cap C_{i,m,r}) \le \sum_{A\in \A_i} \sum_{B\in \A_i}  b_{n-m,B,A} |I_{m,r,B}| \frac{\by_i(A)}{\rho_i^n}\,,
$$
where $b_{k,B,A}$ is the number of different translates of $B$ which appear in the border of $\om^{k}(A)$.

On the other hand, the hypothesis implies that there exists
$n_0>0$ such that for every $A,B\in \A_i$ there exists a translate
of $B$ in the interior of $\om^{n_0}(A)$. Thus, there exists
$0\leq \delta<1$ such that $b_{n_0,B,A}\leq \delta M^{n_0}(B,A)$
for every $A,B\in \A_{i}$. Inductively, we deduce
$b_{kn_0,B,A}\leq \delta^k M^{kn_0}(B,A)$ for every $k>0$ and for
every $A,B\in \A_{i}$. Hence we get
\begin{eqnarray*}
\phi_i^*(C_{i,m+nn_0,r} \cap C_{i,m,r}) & \leq & \delta^n \sum_{A\in
\A_{i}}\sum_{B\in\A_{i}}M^{nn_0}(B,A)|I_{m,r,B}|\frac{\by_i(A)}{\rho_i^{m+nn_0}}\\
    & = & \delta^n\sum_{B\in\A_{i}}|I_{m,r,B}| \frac{\by_i(B)}{\rho_i^m},
\end{eqnarray*}
which implies that $\lim_{n\to\infty}\phi_i^*(C_{i,m+nn_0,r} \cap
C_{i,m,r})=0$ and then $\phi_i^*(C)=0$.
\end{proof}
In the sequel we will suppose that the  hypothesis of Lemma
\ref{hip} holds. That is, we will assume that
\begin{equation}
\label{hyp-hip} \exists \,A\in \A_i, \, \exists\, n>0 \mbox{ such
that a translate of } A \mbox{ appears in the interior of }
\omega^n(A).
\end{equation}
\begin{remark}
{\rm
If $\A_i$ is one of the maximal components, that is, $A\in \A_i$ does not appear in the substitution of any prototile from another component, then
(\ref{hyp-hip}) holds automatically, because the admissibility assumption implies that $A$ must appear in the interior of $\om^n(E)$ for some tile
$E$, which can only be from $\A_i$. Lemma \ref{hip} implies that in this case, the set
$C$ is always negligible with respect to $\phi_i^*$.}
\end{remark}

\begin{lemma}
\label{medida0} Suppose that $\omega$ verifies (\ref{hyp-hip}).
Then for every non-empty open set $U\subseteq \Gamma_i$ there
exists a countable collection $(C_n)_{n\in\NN}\subseteq \F_i$ of
disjoint sets such that $\bigcup_{n\in\NN}C_n\subseteq U$ and
$\phi_i^*(U\setminus \bigcup_{n\in\NN}C_n)=0$.  Thus
$\B(\Gamma_i)\subseteq \eta_i^*$.
\end{lemma}
\begin{proof}
Let  $U_1\subseteq U$ the set of all $\T\in U$ for which there
exist $n_{\T}\in\NN$, $A_{\T}\in \A_{\J^c_i}$ and $\bv_{\T}\in
J_{A_{\T}}^{(n_{\T})}$ such that
$$
\T\in \omega^{n_{\T}}(C_{A_{\T}})-\bv_{\T}\subseteq U.
$$
Note that $U\setminus U_1\seq C$ by (\ref{eq-goodint}).
We have
$$
U_1\subseteq \bigcup_{\T\in
U_1}\omega^{n_{\T}}(C_{A_{\T}})-\bv_{\T}\subseteq U,
$$
and since the collection $\F_i$ is countable, there exists a
sequence $(\T_n)_{n\in\NN}\subseteq U_1$ such that
$$
U_1\subseteq\bigcup_{\T\in
U_1}\left(\omega^{n_{\T}}(C_{A_{\T}})-\bv_{\T}\right)=\bigcup_{k\in\NN}\left(\omega^{n_{\T_k}}(C_{A_{\T_k}})-\bv_{\T_k}\right)\subseteq
U.
$$
Moreover, thanks to Lemma \ref{partition}, the sets
$(\omega^{n_{\T_k}}(C_{A_{\T_k}})-\bv_{\T_k})_{k\in\NN}$ can be
chosen disjoint.
Since
$$
U\setminus
\bigcup_{k\in\NN}\left(\omega^{n_{\T_k}}(C_{A_{\T_k}})-\bv_{\T_k}\right)\subseteq
U\setminus U_1\subseteq C,
$$
Lemma \ref{hip} implies that $U$ is in $\eta_i^*$, because it is a countable
union of sets in $\F_i$ up to a negligible set with respect to
$\phi_i^*$.
\end{proof}

Now define $\mu^T_i: \B(\Gamma)\to \overline{\RR}_+$ by
$$\mu^T_i(U)=\phi_i^*(U\cap \Gamma_i) \mbox{ for every } U\in \B(\Gamma).$$
\begin{lemma}
\label{medida1} Suppose that $\omega$ verifies (\ref{hyp-hip}).
Then $\mu^T_i$ is a $\sigma$-finite transverse measure on
$\B(\Gamma)$ supported on $Y_i\cap \Gam$. Furthermore,
$$
\mu^T_i(C_A)=\by_i(A) \mbox{ for every } A\in\A.
$$
\end{lemma}
\begin{proof}
Lemma \ref{medida0} ensures that the restriction of $\phi_i^*$ to
$\B(\Gamma_i)$ is a measure. Then $\mu^T_i$ is a measure on
$\B(\Gamma)$. It is $\sig$-finite because $\mu_i^T(\Gamma\setminus
\Gamma_i)=0$ and   $\Gamma_i$ is a countable union of sets of
finite measure. Lemma \ref{supported-on-i} implies that $\mu_i^T$
is supported on $Y_i\cap \Gam$.

Next we prove that $\mu_i^T$ is transverse. Let $U\in \B(\Gamma)$
and $\bv\in\RR^d$ be such that $U-\bv\subseteq \Gamma$.

\medskip

\underline{\em Step 1.} First we suppose that
$U=\omega^n(C_A)-\bu$, for some  $A\in \AJC\cap \A'$,
$\bu\in J_A^{(n)}$ and $n\geq 0$. If $\bu+\bv\in
J_A^{(n)}$ then by definition of $\mu_i^T$ we have
$\mu^T_i(U)=\mu^T_i(U-\bv)$. If not, for $m>n$ consider the
sets $\omega^m(C_{A_{m,1}})-\bu_{m,1}, \cdots,
\omega^m(C_{A_{m,k_m}})-\bu_{m,k_m}$ in $\F_i$ whose union is
equal to $U$ (this corresponds to looking at $m$-level supertiles
and finding translates of $\om^n(A)$ in them). This union is
disjoint since $A\in \A'$. Let $J_m=\{1\leq i \leq k_m:
u_{m,i}+\bv \in J_{A_{m,i}}^{(m)}\}$ and $U_m=\bigcup_{i\notin
J_m}(\omega^m(C_{A_{m,i}})-\bu_{m,i})$. We have
\begin{eqnarray*}
\mu_i^T(U)& = &\sum_{i\in
J_m}\mu_i^T(\omega^m(C_{A_{m,i}})-\bu_{m,i}) +\mu_i^T(U_m),\\
\mu_i^T(U-\bv) & = & \sum_{i\in
J_m}\mu_i^T(\omega^m(C_{A_{m,i}})-\bu_{m,i})+\mu_i^T(U_m-\bv),
\end{eqnarray*}
hence
$$
\mu_i^T(U)-\mu_i^T(U-\bv) = \mu_i^T(U_m) - \mu_i^T(U_m-\bv).
$$
Note that $U_{m+1} \seq U_m$ and $\bigcap_{m} U_m \seq C,\ \ \bigcap_{m} (U_m -\bv)\seq C$, so
Lemma \ref{hip} implies
$\mu^T_i(U-\bv)=\mu^T_i(U)$.
If $A\in \AJC \setminus \A'$, then $A$ is in a minimal component which has no access to $\A_i$ and we have $\mu_i^T(U) = 0$.
Then a similar argument yields $\mu_i^T(U - \bv) \le \mu_i^T(U_m-\bv)$, whence $\mu_i^T(U-\bv)=0$.

\medskip

\underline{\em Step 2.} Now we suppose that $U\subseteq \Gamma_i$
is an open set. Let $(C_n)_{n\in\NN}\subseteq \F_i$ be a disjoint
collection of sets such that $\bigcup_{n\in\NN}C_n\subseteq U$ and
$\mu^T_i(U)=\sum_{n\in\NN}\mu_i^T(C_n)$. This collection exists
due to Lemma \ref{medida0}. On the one hand, Step 1 implies that
$\mu^T_i(U)=\sum_{n\in\NN}\mu^T_i(C_n-\bv)=\mu^T_i(\bigcup_{n\in\NN}(C_n-\bv))$.
On the other hand, $(U-\bv)\setminus
\bigcup_{n\in\NN}(C_n-\bv)\subseteq C$ which implies that
$\mu^T_i(U-\bv)=\mu_i^T(U)$.

\medskip

\underline{\em Step 3.} Now let $U$ be any set in $\B(\Gamma)$.
Since the elements in $\F_i$ are clopen sets in $\Gamma_i$,   we
have
\begin{eqnarray}
\phi_i^*(U\cap\Gamma_i)& = & \inf\left\{\sum_{n\in\NN}\mu_i^T(C_n):
(C_n)_{n\in\NN}\subseteq \F_i\cup\emptyset, U\cap\Gamma_i\subseteq
\bigcup_{n\in\NN}C_n\right\}\nonumber\\
&\geq & \inf\{\mu^T_i(V): U\cap\Gamma_i\subseteq V, \mbox{ and } V
\mbox{ is
open in } \Gamma_i\}. \label{eq-new3}
%\\&\geq & \phi_i^*(U\cap\Gamma_i).
\end{eqnarray}
Let $V$ be an open set in $\Gamma_i$ which contains $U\cap \Gamma_i$.
We have
$$
\mu_i^T((U\cap\Gamma_i)-\bv)\leq \mu_i^T(V-\bv)=\mu_i^T(V),
$$
which implies, by (\ref{eq-new3}),
$$
 \mu_i^T((U\cap\Gamma_i)-\bv)\leq \phi_i^*(U\cap \Gamma_i) = \mu_i^T(U).
$$
We claim that
$\mu_i^T((U\cap\Gamma_i)-\bv)=\mu_i^T(U-\bv)$. Indeed,
$$
\mu_i^T(U-\bv) = \mu_i^T((U\cap \Gamma_i) - \bv) + \mu_i^T((U\cap \Gamma_i^c) - \bv),
$$
but
$$
\mu_i^T((U\cap \Gamma_i^c) - \bv) = \mu_i^T(((U\cap \Gamma_i^c) - \bv)\cap \Gamma_i) = 0,
$$
because $((U\cap \Gamma_i^c) - \bv)\cap \Gamma_i\seq C$.
The claim is proved, and we obtain $\mu_i^T(U - \bv) \le \mu_i^T(U)$.

Finally, replacing $U$ by $U-\bv$ and $\bv$ by $-\bv$ we get
$\mu_i^T(U)=\mu_i^T(U-\bv)$.
This concludes the proof that $\mu_i^T$ is a transverse measure.

\medskip

It remains to verify the formula. By definition,
$\mu_i^T(C_A)=\by_i(A)$ for every $A\in\AJC$. If $A\in \J_i$
then Lemma \ref{coreN}, Lemma \ref{help-lemma} and Lemma
\ref{finite} imply that $\mu_i^T(C_A)=\infty$, which is equal to
$\by_i(A)$.
%If $A$ has no access to $\A_i$ then by definition we have $\mu_i^T(C_A)=0=\by_i(A)$.
\end{proof}

\begin{lemma}\label{conditions-zero-measure}
Let $\mu$ be an invariant $\sigma$-finite measure of the tiling
system $(X,\RR^d)$, and let $U\in \B(X)$ be an invariant set such
that $\mu(U)=0$. Then $\mu^T(U\cap \Gamma)=0$.
\end{lemma}
\begin{proof}
Since $U$ is invariant, for every patch $P$ we have $U\cap
(C_P+B_{\varepsilon}(\b0))=(U\cap C_P)+B_{\varepsilon}(\b0)$. Thus if
$P$ is centered at $\b0$ and $\varepsilon$ is sufficiently small, we
get $$0=\mu(U\cap (C_P+B_{\varepsilon}(\b0)))=\mu^T(U\cap
C_P)\vol(B_{\varepsilon}(\b0))$$
by Lemmas \ref{invariant-transversal} and \ref{transversal-0}. Since $\Gamma=\bigcup_{A\in\A}C_A$,
we deduce $\mu^T(U\cap \Gamma)=0$.
\end{proof}
\begin{lemma}
\label{ergodic} Let $\mu$ and $\nu$ be two $\sig$-finite ergodic invariant
measures for the tiling system $(X,\RR^d)$ for which there exists
$A\in \A$ such that $0<\mu^T(C_A)=\nu^T(C_A)<\infty$ and
$\mu^T|_{\B(C_A)}=\nu^T|_{\B(C_A)}$. Then $\mu=\nu$.
\end{lemma}
\begin{proof}
Let $U$ be the subset of $X$ of all the tilings $\T$ containing a
tile   equivalent to  $A$. This set is open and invariant, and
contains the set $C_A+B_{\varepsilon}(\b0)$, for every
$\varepsilon>0$. Thus, for $\varepsilon$ sufficiently small we get
$\mu(U),\nu(U)\geq \mu^T(C_A)\vol(B_{\varepsilon}(\b0))>0$. This
implies, by the ergodicity of   $\mu$ and $\nu$, that
$\mu(U^c)=\nu(U^c)=0$. Since $U^c$ is invariant, Lemma
\ref{conditions-zero-measure} implies that
$\mu^T(U^c\cap\Gamma)=\nu^T(U^c\cap\Gamma)=0$. Then for every
$V\in \B(\Gamma)$, $\mu^T(V)=\mu^T(V\cap U)$ and
$\nu^T(V)=\nu^T(V\cap U)$.
%$\mu(\tilde{U}),\nu(\tilde{U}\geq
%\mu^T(U)\vol(B_{\varepsilon}(\vec{0}))$. Since $\mu$ and $\nu$ are
%ergodic, this implies that $\mu(\tilde{U}^c)=\nu(\tilde{U}^c)=0$.
%Since $\tilde{U}^c$ is invariant, this also implies that
%$\mu^T(\tilde{U}^c\cap\Gamma)=\nu^T(\tilde{U}^c\cap\Gamma)=0$.
Let $\Lambda=\{\bv\in\RR^d: (\Gamma-\bv)\cap \Gamma\neq
\emptyset\}$. This set is countable because $X$ satisfies the FPC; let
$\Lambda=\{\bv_n: n\geq 0\}$. For $V\in\B(\Gamma)$, we
have
$$
U\cap V=\bigcup_{n\geq 0}(C_A+\bv_n)\cap V=\bigcup_{n\geq 0}
V_n,
$$
where $V_0=V\cap(C_A+\bv_0)$ and
$V_n=(V\cap(C_A+\bv_n))\setminus(V_0\cup\cdots\cup V_{n-1})$,
for $n>0$. Then
$$
\mu^T(V)=\mu^T(V\cap U)=\sum_{n\geq 0}\mu^T(V_n) \hspace{2mm}
\mbox{ and } \hspace{2mm} \nu^T(V)=\nu^T(V\cap U)=\sum_{n\geq
0}\nu^T(V_n).
$$
Since for every $n\geq 0$ we have $V_n-\bv_n\subseteq C_A$, we
get
$$\mu^T(V_n)=\mu^T(V_n-\bv_n)=\nu^T(V_n-\bv_n)=\nu^T(V_n),$$
which implies that $\mu^T(V)=\nu^T(V)$. Theorem
\ref{correspondence} implies $\mu=\nu$.
\end{proof}
\begin{theorem}
\label{final-theorem-infinite} Let $\om$ be an admissible tile substitution, which is partially recognizable and satisfies (\ref{eq-added}) and (\ref{eq-irre}).
\begin{enumerate}
\item[(i)] Let $m+1\leq i\leq l$ be such that
$M_i$ is primitive, and suppose that (\ref{hyp-hip}) holds for $\A_i$. Then the
measure $\mu^T_i$ is the unique transverse measure on $\B(\Gamma)$
supported on $Y_i\cap \Gamma$ such that
$$
\mu^T_i(C_A)=\by_i(A) \mbox{ for every } A\in \A.
$$
Moreover, the associated invariant measure $\mu_i$ is
$\sig$-finite and ergodic. \item[(ii)] Suppose that
(\ref{hyp-hip}) holds for all $\A_i \seq \A'$. Then any
$\sig$-finite ergodic measure $\mu$ with the property that
$0<\mu^T(C_A)<\infty$ for some $A\in\A'$, is equal to some measure
$\mu_i$ up to scaling. \item[(iii)] If $Y_i$ is a maximal
component, then any $\sig$-finite ergodic measure on $Y_i$ which
is positive and finite on some open set, is equal to $\mu_i$ up to
scaling. (Note that (\ref{hyp-hip}) holds for maximal components
by admissibility.)
\end{enumerate}
\end{theorem}
\begin{proof}
(i)
The uniqueness of $\mu_i^T$ follows from the fact that if $\nu^T$
is another transverse measure satisfying
$\nu^T|_{\F_i}=\mu_i^T|_{\F_i}$ then
$\nu^T|_{\B(\Gamma_i)}=\mu^T_i|_{\B(\Gamma_i)}$. Let $\mu_i$ be
the invariant measure of $(X_{\A,\omega},\RR^d)$ associated to
$\mu_i^T$. It is $\sig$-finite and invariant by Theorem~\ref{correspondence}. It remains to verify that it is ergodic.

Let $U\in \B(X_{\A,\omega})$ be an invariant set, that is,
$U-\bv=U$ for every $\bv\in\RR^d$.  Since the set $U$ is invariant
the measure $\mu_{U} =\mu_i|_{U}$ is invariant. Let $\mu_{U}^T$ be
the transverse measure on $\B(\Gamma)$ associated to  $\mu_{U}$.
By  Lemma  \ref{invariant-transversal}, we get
$$
\mu_{U}^T(C)=\mu_i^T(C\cap U) \mbox{ for every } C\in\B(\Gamma).
$$
The measure $\mu_{U}^T$ verifies $\mu_{U}^T\leq \mu_i^{T}$,  which
ensures that $\mu_{U}^T$ is supported on $\Gamma_i$, and that
$\mu_U(C_B)<\infty$ for every $B\in \J_i^c$. In particular,
$\mu_U^T(C_B)=0$ for every $B\in (\J_i\cup \I_i)^c.$

{\it Case 1:} there exists $A$ in the class  $\A_i$ such that
$\mu_{U}^T(C_A)>0$. Then $\mu_{U}^T(C_B)>0$ for every $B$ in the
class $\A_i$. Thus Lemma \ref{help-lemma} implies that
$\mu_U^T(C_B)=\infty$ for every $B\in \J_i$, and  since
$\mu_U^T(C_B)=0$ for every $B\in (\J_i\cup \I_i)^c,$ the vector
$(\mu_U^T(C_B))_{B\in\I_i}$ is in the core of the restriction of
$M$ to $\I_i$. This implies that $(\mu^T_{U}(C_B))_{B\in\A}=\alpha
\by_i$, for some $0<\alpha\leq 1$. This vector determines
$\mu_U^T(\omega^n(C_B)-\bv)$, for every $B\in\J_i^c$, $\bv\in
J_B^{(n)}$ and $n\geq 0$. Then
$\alpha\mu^T_i|_{\F_i}=\mu^T_{U}|_{\F_i}$, and since $\mu^T_{U}$
is supported on $\Gamma_i$, we get that $\mu^T_{U}=\alpha\mu^T_i$.
This implies that $\mu^T_i(U^c\cap \Gamma)=0$, because of
$\mu^T_U(U^c\cap \Gamma)=0$. This shows that $\alpha=1$ and then
$\mu_U^T=\mu_i^T$. From Theorem \ref{correspondence} we obtain
that $\mu_U=\mu_i$, which implies that $\mu_i(U^c)=0$.

{\it Case 2:} If $\mu^T_{U}(C_A)=0$ for every $A$ in the class
$\A_i$ then  $\mu_{U^c}^T(C_A)=\mu_i^{T}(C_A)$ for every $A$ in
the class $\A_i$. As in the previous case, but replacing $U$ by
$U^c$, we get  that $\mu^T_{U^c}=\mu_i^T$ and $\mu_i(U)=0$. This
shows that $\mu_i$ is ergodic.

\medskip

(ii)
Let $\mu$ be  a $\sig$-finite ergodic measure such that
$0<\mu^T(C_A)<\infty$ for some $A\in\A'$. Let $i=\max\{1\leq j\leq
l: 0<\mu^T(C_A)<\infty, A\in\A_j\}.$  By Lemma \ref{coreN}, the
vector $(\mu^T(C_A))_{A\in\A_i}$ is in $core(M_i)$. Then there
exists $\lambda>0$ such that for every $A\in\A_i$, $\bv\in
J_A^{(n)}$ and $n\geq 0$,
$$
\mu^T(\omega^n(C_A)+\bv)=\lambda\mu_i^T(\omega^n(C_A)+\bv).
$$
A standard argument shows that this equation implies
that $\mu_T$ and $\lambda\mu_i^T$ coincide on each Borel set
contained in $\bigcup_{A\in\A_i}C_A$. From Lemma \ref{ergodic} we
obtain that $\mu=\lambda\mu_i$.

\medskip

(iii)
Let $\mu$ be an ergodic $\sig$-finite measure on a maximal component $Y_i$, such that $\mu(U)$ is positive and finite for some open set $U\seq Y_i$. Let $\mu^T$ be the corresponding transverse measure; then $\mu^T(U\cap \Gamma)$ is positive and finite. The topology on $Y_i\cap \Gamma$ is generated by the sets $C_P$ for $P \in \Lambda_X$ (see Section 2.7), with $P$ containing at least one
tile from $\A_i$. Decomposing $C_P$ as a disjoint union, we can find $A\in \A_i$ and $n\in \NN$ such that $\mu^T(\om^n(C_A)-\bv)$ is positive and finite for some $\bv\in \R^d$ such that $\om^n(C_A)-\bv\in \Gamma$.
Then Lemma~\ref{coreN} implies that $\mu^T(C_A)$ is positive and finite, and we can conclude by applying part (ii).
%If $Y_i$ is a maximal component, then $Y_i\cap Y_j\neq \emptyset$,
%for every $m+1\leq j\neq i \leq l$. This implies that the measures
%supported on $Y_i$ are equal to $\mu_i$ up to scaling.
\end{proof}

\medskip \noindent {\em Proof of Theorem B.} This follows from Theorem~\ref{final-theorem-infinite}. We only need to note that every point in a maximal component $Y_i$ has a neighborhood of the form
$(\om^n(C_A)-\bv)+B_\varepsilon(\b0)$, with $A\in \A_i$, which has positive and finite $\mu_i$ measure. \qed

%%%%%%%%%%%%%%%%%%%%%%%%%%%%%%%%%%%%%%%%%%%%%%%%%%%%%%%%%%%%%%%%%%%%%%%%%%%%%%%%%%
\section{Examples and concluding remarks}

\begin{example}\label{nonperiodic-substitution}{\em
All the tiles have the unit square as its support and are distinguished only by the labels. Let $\Ak = \left\{\begin{tabular}{|c|} \hline 0 \\ \hline \end{tabular}\,,
\begin{tabular}{|c|} \hline 1 \\ \hline \end{tabular}\,,\begin{tabular}{|c|} \hline 2 \\ \hline \end{tabular}\right\}$;
$$
\begin{tabular}{|c|} \hline 0 \\ \hline \end{tabular}\ \to\  \begin{tabular}{|c|c|c|} \hline 0 & 0 & 1 \\ \hline
                                                                                             0 & 0 & 1 \\ \hline
                                                                             1 & 1 & 1 \\ \hline
                    \end{tabular}\,, \ \ \ \ \ \
  \begin{tabular}{|c|} \hline 1 \\ \hline \end{tabular}\  \to\  \begin{tabular}{|c|c|c|} \hline 1 & 1 & 0 \\ \hline
                                                                                               1 & 1 & 0 \\ \hline
                                                                     0 & 0 & 0 \\ \hline \end{tabular}\,,  \ \ \ \ \ \
\begin{tabular}{|c|} \hline 2 \\ \hline \end{tabular}\  \to\  \begin{tabular}{|c|c|c|} \hline 1 & 1 & 1 \\ \hline
                                                                                              1 & 2 & 2 \\ \hline
                                                      1 & 2 & 2 \\ \hline \end{tabular}
$$
The substitution matrix is $M = \left(\begin{array}{ccc} 4 & 5 & 0 \\ 5 & 4 & 5 \\ 0 & 0 & 4 \end{array}\right)$. The tiling dynamical system has
one minimal and one maximal component. It is easy to check that it is non-periodic.
The  restriction of the
substitution matrix to the minimal components is
$
\left(%
\begin{array}{cc}
  4 & 5 \\
  5 & 4 \\
\end{array}%
\right).
$
Then the unique   probability measure $\mu$ is given by
$\mu^T(C_2)=0$ and
$$
\mu^T(\omega^n(C_0)-\bv)=\mu^T(\omega^n(C_1)-\bv)=\frac{1}{2\cdot
9^n}, \mbox{ for every } \bv\in J_0^ {(n)}=J_1^{(n)} \mbox{
and } n\geq 0.
$$
This substitution satisfies (\ref{hyp-hip}); then applying Theorem
\ref{final-theorem-infinite} we get that every $\sig$-finite
ergodic measure $\mu$ such that $0<\mu^T(C_{2})<\infty$, is a constant
multiple of the unique measure $\mu_2$ such that $\mu_2^T$ is
supported on $\bigcup_{n\geq 0}\bigcup_{\bv\in
J_2^{(n)}}(\omega^n(C_2)-\bv)$ and that verifies
$$
\mu_2^{T}(C_2)=1 \mbox{ and } \mu_2^T(C_0)=\mu_2^T(C_1)=\infty.
$$
Since the restriction of the substitution matrix to $\A'$ is equal
to  $[4]$, we get
$$
4^{-n}\mu_2^T(C_2)=\mu_2^T(\omega^n(C_2)-\bv),
$$
for every $\bv\in J_2^{(n)}$ and $n\geq 0$.}
\end{example}

\begin{example}[Sierpi\'nski carpet]\label{Cantor-substitution}{\em
All the tiles have the unit square as its support and are distinguished only by the labels. Let $\Ak = \left\{\begin{tabular}{|c|} \hline 0 \\ \hline \end{tabular}\,,
\begin{tabular}{|c|} \hline 1 \\ \hline \end{tabular}\right\}$;
$$
\begin{tabular}{|c|} \hline 0 \\ \hline \end{tabular}\ \to\  \begin{tabular}{|c|c|c|} \hline 0 & 0 & 0 \\ \hline
                                                                                             0 & 0 & 0 \\ \hline
                                                                                             0 & 0 & 0 \\ \hline
                                                                                                                                  \end{tabular}\,, \ \ \ \ \ \
                    \begin{tabular}{|c|} \hline 1 \\ \hline \end{tabular}\  \to\  \begin{tabular}{|c|c|c|} \hline 1 & 1 & 1 \\ \hline
                                                                              1 & 0 & 1 \\ \hline
                                                                                                                              1 & 1 & 1 \\ \hline \end{tabular}
$$
The substitution matrix is $M = \left(\begin{array}{cc} 9 & 1 \\ 0 & 8 \end{array} \right)$.
Here the minimal component is periodic; it consists of periodic tilings with only one tile type, labeled $0$. However, the ``non-periodic border''
condition holds (see Section 4 for details).
This example, which we call the ``integer Sierpi\'nski carpet'' tiling, is a generalization of the 1-dimensional symbolic substitution $0\to 000,\ 1\to 101$, which
was analyzed by A. Fisher \cite{Fi1}.
The intersection of the
transversal with the unique minimal component contains only one
element $\{\T_{0}\}$. Then the transverse measure associated to
the unique invariant probability measure $\mu_0$ supported on the
minimal component is the atomic measure $\mu_0^T(\{\T_0\})=1$. The
measure $\mu_0$ corresponds to the Lebesgue measure on the torus
(the minimal component is conjugate to the $\RR^2$-translations on
the torus).
This substitution satisfies (\ref{hyp-hip}), then applying Theorem
\ref{final-theorem-infinite} we get that every $\sig$-finite
ergodic measure $\mu$ such that $0<\mu^T(C_{1})<\infty$, is a constant
multiple of the unique measure $\mu_1$ such that $\mu_1^T$ is
supported on $\bigcup_{n\geq 0}\bigcup_{\bv\in
J_1^{(n)}}(\omega^n(C_1)-\bv)$ and verifies
$$
\mu_1^{T}(C_1)=1 \mbox{ and } \mu_1^T(C_0)=\infty.
$$
Since the restriction of the substitution matrix to $\A'$ is equal
to $[8]$, we get
$$
8^{-n}\mu_1^T(C_1)=\mu_1^T(\omega^n(C_1)-\bv),
$$
for every $\bv\in J_1^{(n)}$ and $n\geq 0$.}
\end{example}

\begin{example}[Sierpi\'nski gasket]{\em
Consider $\A=\{\vartriangle, \triangledown,
\blacktriangle\}$ and the substitution given below.
$$
\blacktriangle\ \to\ \begin{array}{ccc} & \!\!\!\!\!\!\!\!\blacktriangle & \\[-9pt] \blacktriangle & \!\!\!\!\!\!\!\!\triangledown & \!\!\!\!\!\!\!\!\!\!\!\!\blacktriangle \end{array},\ \ \
\vartriangle\ \to\ \begin{array}{ccc} & \!\!\!\!\!\!\!\!\vartriangle & \\[-9pt] \vartriangle & \!\!\!\!\!\!\!\!\triangledown & \!\!\!\!\!\!\!\!\!\!\!\!\vartriangle \end{array},\ \ \
\triangledown\ \to \begin{array}{ccc} \triangledown & \!\!\!\!\!\!\!\!\vartriangle & \!\!\!\!\!\!\!\!\!\!\!\!\triangledown \\[-9pt] & \!\!\!\!\!\!\!\!\triangledown & \end{array}
$$
The
substitution matrix of $\omega$ is
$
M=\left(%
\begin{array}{ccc}
  3 & 1 & 0 \\
  1 & 3 & 1\\
  0 & 0 & 3
\end{array}%
\right).
$
The system $X_{\A,\omega}$ has a unique minimal component, and the
submatrix of $M$ associated to the unique minimal component is
$\left(%
\begin{array}{cc}
  3 & 1  \\
  1 & 3 \\
\end{array}%
\right).$ This implies that the unique probability transversal
measure is given by
$$
\mu^T(\omega^n(C_{\vartriangle})-\bv)=2^{-2n-1}
\mbox{ and }
\mu^T(\omega^n(C_{\triangledown})-\bu)=2^{-2n-1},
$$
where  $\bv\in J_{\vartriangle}^{(n)}$, $\bu\in
J_{\triangledown}^{(n)}$ and $n\geq 0$.

The tile substitution satisfies the non-periodic border condition, so it is partially recognizable (this is also easy to verify directly) and satisfies (\ref{hyp-hip}).
By Theorem B, there is a unique, up to scaling, ergodic $\sig$-finite measure $\mu_\blacktriangle$, for which every point containing a tile $\blacktriangle$ has a neighborhood with positive and finite measure.
For this measure we have
$$
3^{-n}\mu_\blacktriangle^T(C_{\blacktriangle})=\mu_\blacktriangle^T(\omega^n(C_{\blacktriangle})-\bv),
$$
for every $\bv\in J_{\blacktriangle}^{(n)}$ and $n\geq 0$, and
$\mu_\blacktriangle^T(C_{\vartriangle})=\mu_\blacktriangle^T(C_{\triangledown})=\infty$.}
\end{example}

\begin{example} \label{ex-saf} {\em
Let $\Ak = \left\{\begin{tabular}{|c|} \hline 0 \\ \hline \end{tabular},
\begin{tabular}{|c|} \hline 1 \\ \hline \end{tabular}\right\}$;
$$
\begin{tabular}{|c|} \hline 0 \\ \hline \end{tabular}\ \to\  \begin{tabular}{|c|c|c|} \hline 0 & 0 & 0 \\ \hline
                                      0 & 0 & 0 \\ \hline
\end{tabular}\,
,\ \ \ \ \ \
\begin{tabular}{|c|} \hline 1 \\ \hline \end{tabular}\  \to\  \begin{tabular}{|c|c|c|} \hline 0 & 1 & 0 \\ \hline
                                                                                1 & 0 & 1 \\ \hline
                                     \end{tabular}
$$
This substitution is {\em self-affine}, rather than self-similar, and generates the integer analog of the ``Bedford-McMullen carpet'', see
\cite{Bedf,McM}. Note that here the non-periodic border condition does not hold; however, partial recognizability (i.e.\ every tiling containing
a prototile labeled by 1 has a unique pre-image under the substitution) is easy to verify directly. Thus Theorem~\ref{final-theorem-infinite} applies,
and we get a conclusion similar to the above examples.
}
\end{example}

\subsection{Concluding remarks.}
\begin{enumerate}
\item[1.] One can draw an (admittedly vague) analogy between substitution tiling flows and horocycle flows on manifolds of negative curvature. Moreover, the dynamics of the substitution $\omega$ is analogous to the
geodesic flow. The compact manifold case corresponds to the case of minimal/primitive substitution systems, for which the horocycle flow, respectively, the tiling flow, is uniquely ergodic. The
non-primitive substitutions then loosely correspond to the non-compact (but, perhaps, geometrically finite) case, where often the only ``natural'' invariant measure is $\sig$-finite, see e.g.\ \cite{Burger}.

\item[2.] A. Fisher \cite{Fi1} obtained a ``second-order ergodic theorem'' for his ``integer Cantor set'' substitution system. Is it possible to obtain similar results for our systems? We believe that it is, at least for examples such as the ``Sierpi\'nski gasket and carpet'' tilings. What about more general non-primitive substitutions? One should probably stick to the self-similar case, or else use averaging over the F{\o}lner sets $\varphi^n (B_1(\b0))$. In general, there will be a {\em graph-directed Iterated Function System} associated to the tiling system. Objects analogous to the tilings, such as the ``Sierpi\'nski gasket and carpet'' tilings, were considered by Strichartz, in the framework of {\em Reverse Iterated Function Systems} \cite{Str}.

\item[3.] All our examples have tiles of very simple geometry, but there exist tile substitutions  with very complicated tiles: e.g.\ tiles with a fractal boundary, disconnected tiles, connected tiles with disconnected interior, etc. This is known for primitive substitution tilings, see e.g. \cite{Vince}, and it is easy to construct a non-primitive substitution tiling using the same shapes as for a primitive one. For example, suppose we have a primitive substitution tiling $\om: \A\to \A$, with $\om(A) = \bigcup_{B\in \A} (B + D_B)$ where $D_B$ is a finite set for every prototile $B$.
    Let $\wt{\A} = \A \times \{0,1\}$, in other words, we consider ``old'' prototiles with an additional label. We assume that $\supp(A,j) = \supp(A)$ for $j=1,2$. Consider a non-trivial partition $\A = \A_1 \cup \A_2$ and
    define the substitution $\wt{\omega}:\,\wt{\A} \to \wt{\A}$ by
    $$\wt{\om}(A,0) = \bigcup_{B\in \A} ((B,0) + D_B),\ \ \ \wt{\om}(A,1) = \bigcup_{B\in \A_1} ((B,0) + D_B) \cup \bigcup_{B\in \A_2} ((B,1) + D_B).
    $$
    This is a non-primitive tile substitution, and one can refine this construction to satisfy any additional properties, such as admissibility, non-periodic border, etc.
\end{enumerate}
%%%%%%%%%%%%%%%%%%%%%%%%%%%%%%%%%%%%%%%%%%%%%%%%%%%%%%%%%%%%%%%%%%%%%%%%%%%%%%%%%%%%%%%%%%%%%%%%%%%%%%%%
\section{Appendix: Invariant measures versus transverse
measures.}\label{Apendix A}
We use the notation and terminology from Section~\ref{transversal}.
Recall that
a {\it transverse measure } on $\B(\Gamma)$   is a measure $\mu:
\B(\Gamma)\to \overline{\RR}_+$ such that $\mu(A)=\mu(A-\bv)$,
for every  $A\subseteq \B(\Gamma)$ and $\bv\in\RR^d$ for which
$A-\bv\subseteq \Gamma$ (see \cite[Definition~5.1]{BBG} for a
definition of transverse measure in the context of laminations).
Recall that $\eta>0$ is such that the closure of $B_\eta(\b0)$ is contained in the interior of every prototile.
Observe that
\be \label{eq-inter}
\T\in \Gam,\ \T+\bv\in \Gam,\ \bv\in B_{2\eta}(\b0)\ \Rightarrow\ \bv=\b0.
\ee
We write $X:= \Xa$ to simplify the notation.
\subsection{From invariant measures to transverse measures.}
\begin{lemma}
\label{invariant-transversal} Let $\mu$ be an invariant measure of
$(X,\RR^d)$. For every $U\in\B(\Gamma)$, there exists $\mu^T(U)\in
\overline{\RR}_+$ such that for every open set $\Theta$ contained
in the ball $B_{\eta}(\b0)$, we have
$$
\frac{\mu(U+\Theta)}{\vol(\Theta)}=\mu^T(U).
$$
\end{lemma}
\begin{proof}
Fix $U\seq \B(\Gam)$.
Observe that $\mu_U:\, E\mapsto \mu(U+E)$ is a Borel measure on the ball $\Bl$. This follows from (\ref{eq-inter}), which implies
$$
E_1, E_2 \seq \Bl,\ E_1\cap E_2 = \es\ \Rightarrow\ (U+E_1) \cap (U+E_2) = \es.
$$
Moreover, by the invariance of $\mu$ we have
$$
E \seq \Bl,\ E-\bv \seq \Bl\ \Rightarrow \ \mu(E-\bv) = \mu(E).
$$
It is easy to see that if for some open $\Th_1\seq \Bl$ we have $\mu(U+\Th_1)=0$, then  $\mu(U+\Th)=0$ for all open subsets
of $\Bl$ and we can set $\mu^T(U)=0$. Similarly,  if for some open $\Th_1\seq \Bl$ we have $\mu(U+\Th_1)=\infty$, then  $\mu(U+\Th)=\infty$ for all open subsets of $\Bl$  and we can set $\mu^T(U)=\infty$.
So we can suppose that $\mu_U$ is positive and finite on open subsets of $\Bl$. Consider the restriction of $\mu_U$ to a cube
$\prod_{i=1}^d [a_i, a_i+h)$ contained in $\Bl$ and extend it to $\R^d$ by periodicity, i.e.\ let
$$
\nu_U(E) = \sum_{\bx \in \ZZ^d} \mu\left(U + \left(E \cap \left(\prod_{i=1}^d [a_i, a_i+h) + h\bx\right)\right)\right).
$$
This is a Borel measure on $\R^d$ which is translation-invariant and positive and finite on open subsets.  It follows that $\nu_U(E) = c_U \vol(E)$
and we can set $\mu^T(U) = c_U$.
\end{proof}
\begin{lemma}
\label{transversal-0} Let $\mu^T:\B(\Gamma)\to \overline{\RR}_+$
be the function obtained in Lemma \ref{invariant-transversal}.
Then $\mu^T$ is a transverse measure.
\end{lemma}
\begin{proof}
It is clear that $\mu^T$ is a measure. Indeed, if
$(U_n)_{n\in\NN}$ is a collection of disjoint sets in
$\B(\Gamma)$, and $\varepsilon>0$ is small enough, then the sets
$(U_n+B_{\varepsilon}(\b0))_{n\in\NN}$ are disjoint. It follows from
the definition of $\mu^T$ that
$\mu^T(\bigcup_{n\in\NN}U_n)=\sum_{n\in\NN}\mu^T(U_n)$. If
$U\in\B(\Gamma)$ and $\bv\in\RR^d$ is such that
$U-\bv\subseteq \Gamma$, then for $\varepsilon>0$  we have
$\mu(U-\bv+B_{\varepsilon}(\b0))= \mu(U-B_{\varepsilon}(\b0))$,
which implies that $\mu^T$ is transverse.
\end{proof}
\begin{definition}
Let $\mu$ be an invariant measure of $(X,\RR^d)$. We denote by $\mu^T$
the transverse measure associated to $\mu$.
\end{definition}

\subsection{From transverse measures to invariant
measures}\label{transversal-to-invariant}

Let $\nu$ be a $\sigma$-finite transverse measure on $\Gamma$. We
write $\lambda_d$ for the Lebesgue measure on $\RR^d$, and
$\nu\otimes\lambda_d$ for the product measure on $\Gamma\times \RR^d$.

For every $\bw, \bv\in \RR^d$, we define
$\psi^{(\bw,\bv)}:X\times \RR^d\to X\times \RR^d$ by
$$
\psi^{(\bw,\bv)}(\T,\bu)=(\T-\bw,\bu-\bv),
\mbox{ for every } \T\in X \mbox{ and } \bu\in\RR^d.
$$
This function is a homeomorphism (with respect to the product
topology).
\begin{lemma}\label{transverse-invariant-invariance}
Let $U$ be a Borel set in $\Gamma\times \RR^d$. Then
$$
\nu\otimes\lambda_d(\psi^{(\bw,\bv)}(U))=\nu\otimes\lambda_d(U),
$$
for every $(\bw,\bv)\in\RR^{2d}$ such that
$\psi^{(\bw,\bv)}(U)\subseteq \Gamma\times \RR^d$.
\end{lemma}

\begin{proof}
For every  $U$ in $X\times \RR^d$ and $\bx\in\RR^d$  we set
$$
(U)^{1}(\bx)=\{\T\in X: (\T,\bx)\in U\}.
$$
Let $U$ be a Borel set in $\Gamma\times \RR^d$ and let
$(\bw,\bv)\in\RR^d$ be such that
$\psi^{(\bw,\bv)}(U)\subseteq \Gamma\times\RR^d$. We have
\begin{eqnarray*}
\nu\otimes\lambda_d(\psi^{(\bw,\bv)}(U))& =&
\int\nu((\psi^{(\bw,\bv)}(U))^{1}(\bx))\,d\lambda_d(\bx)\\
&=&\int\nu(U^1(\bx+\bv)-\bw)\,d\lambda_d(\bx)\\
\end{eqnarray*}
Since $\nu$ is transverse and $U^1(\by)-\bw\subseteq
\Gamma$, for every $\by\in\RR^d$, we get
$$
\int\nu(U^1(\bx+\bv)-\bw)\,d\lambda_d(bx) =
\int\nu(U^1(\bx+\bv))\,d\lambda_d(\bx).
$$
The invariance under translations of the Lebesgue measure implies
that
$$
\int\nu(U^1(\bx+\bv))\,d\lambda_d(\bx)=\int\nu(U^1(\bx))\,d\lambda(\bx)=\nu\otimes\lambda_d(U).
$$
\end{proof}

Since $X$ verifies FPC, $X$ is a finite union of sets
$C_P+\Theta$, with $P\in\Lambda_X$ and $\Theta$ open in $\RR^d$
with diameter smaller than $\eta$. Namely, $X=\bigcup_{i=1}^nU_i$,
where $U_i=C_{P_i}+\Theta_i$, for every $1\leq i\leq n$.

From (\ref{eq-inter}), for each $1\leq i\leq n$ the function
$h_i:U_i\to C_i\times \Theta_i$ given by
$h_i(\T+\bv)=(\T,\bv)$ is well-defined. Moreover, $h_i$ is
a homeomorphism.

For every $1\leq i\leq n$, let $\ba_i\in\RR^d$ be a vector
such that $\b0\in \Theta_i-\ba_i$ and $\Theta_i-\ba_i$ is
contained in the ball $B_{\eta}(\b0)$. Since for every $R>0$, the
set $C_{P_i}$ is a finite and disjoint union of sets $C_P$, with
$P\in\Lambda_X$ whose support contains the ball $B_R(\b0)$, we
can assume that the support of $C_{P_i}$ contains the vectors
$\ba_k-\ba_j$, for every $1\leq k,j\leq n$. This implies
that for every $1\leq i,j\leq n$, there exist $\ba_{i,j},
\bB_{i,j}\in\RR^d$ such that
\begin{equation}
\label{eq-transverse-invariant-1} h_j\circ
h_i^{-1}(\T,\bv)=(\T+\ba_{i,j},\bv+\bB_{i,j}),
\mbox{ for every } \T\in C_{P_i}, \bv\in\Theta_i.
\end{equation}

Equation (\ref{eq-transverse-invariant-1}) implies that $X$ has a
{\it $d$-lamination structure} (see \cite{BBG} for details). The
collection $\{U_i,h_i\}_{i=1}^n$ is called  an {\it atlas} of $X$.

\medskip

For every $1\leq i\leq n$ define $\mu_i:\B(U_i)\to
\overline{\RR}_+$ by $ \mu_i(U)=\nu\otimes\lambda_d(h_i(U))$. It
is clear that $\mu_i$ is a measure.

We define $\wt{U}_1=U_1$ and
$\wt{U}_i=U_i\setminus(\bigcup_{j=1}^{i-1}U_j)$, for every
$2\leq i\leq n$. The function
$\mu=\sum_{i=1}^n\mu_i|_{\wt{U}_i}$ is a measure on $\B(X)$.
Since $\nu\otimes\lambda_d$ is $\sig$-finite, $\mu$ is $\sig$-finite
too.
We will show that $\mu$ is invariant and that $\mu^T=\nu$.

%\begin{lemma}
%The measures $\mu_i$ and $\mu_j$ coincide on $U_i\cap U_j$.
%\end{lemma}

\begin{lemma}\label{transverse-invariant-independence}
The measure $\mu$ does not depend on the atlas.
\end{lemma}
\begin{proof}
Since $X$ verifies FPC, the set $\Lambda=\{\bv\in\RR^d:
(\Gamma-\bv)\cap \Gamma\neq \emptyset\}$ is countable. Let
$\Lambda=\{\bv_n: n\in\NN\}$.

Let  $\{V_i,f_i\}_{i=1}^m$ be another atlas of $X$. For every
$1\leq i\leq m$, let $\wt{\mu}_i$ be the measure on $V_i$
defined as $\wt{\mu}_i=(\nu\otimes\lambda_d)\circ f_i$.  We set
$\wt{V}_1=V_1$ and
$\wt{V}_i=V_i\setminus(\bigcup_{j=1}^{i-1}V_j)$, for every
$2\leq i\leq m$. Denote by $\wt{\mu}$ the measure on $X$ defined
by  $\wt{\mu}=\sum_{i=1}^m\wt{\mu}_i|_{\wt{V}_i}$.

Let $\T$ be a tiling in $U_i\cap V_j$, for some $1\leq i\leq n$
and $1\leq j\leq m$. If $h_i(\T)=(\T_i,\bv_i)$ then there
exists $\bv(\T)\in\Lambda$ such that $f_j(\T)=(\T_i+\bv(\T),
\bv_i-\bv(\T))$. Since $\bv(\T)$ is in $\Lambda$,
there exists $n\in\NN$ such that $\bv(\T)=\bv_n$. Thus if $U$
is a Borel set in $U_i\cap V_j$, it can be written as
$$
U=\bigcup_{n\in\NN} U_n, \mbox{ where } U_n=\{\T\in U:
\,f_j(\T)=\psi^{\bv_n}(h_i(\T))\},
$$
where $\psi^{\bv_n}$ abbreviates
$\psi^{(\bv_n,-\bv_n)}$. The sets $U_n$ are disjoint and
measurable.

From Lemma \ref{transverse-invariant-invariance}, for every
$n\in\NN$ we have
\begin{eqnarray*}
\mu_i(U_n)&=& \nu\otimes\lambda_d(h_i(U_n))\\
          &=& \nu\otimes\lambda_d(\psi^{\bv_n}(h_i(U_n)))\\
          &=&\nu\otimes\lambda_d(f_j(U_n))\\
          &=&\wt{\mu}_j(U_n),
\end{eqnarray*}
which implies that $\mu_i(U)=\wt{\mu}_j(U)$, and hence
$\mu=\wt{\mu}$.
\end{proof}

\begin{remark}
{\rm From the proof of Lemma
\ref{transverse-invariant-independence} we also deduce that
$\mu(U)=\mu_i(U)$, for every Borel set $U\subseteq U_i$. }
\end{remark}

\begin{lemma}\label{transverse-to-invariant}
The measure $\mu$ is invariant and $\mu^T=\nu$.
\end{lemma}
\begin{proof}
Let $\bv\in\RR^d$. For every $1\leq i\leq n$, let
$V_i=U_i-\bv$ and $f_i:V_i\to C_{P_i}\times(\Theta_i-\bv)$
be defined as $f_i(\T)=\psi^{(\b0,\bv)}(h_i(\T+\bv))$.
The collection $\{V_i,f_i\}$ is an atlas of $X$.

Let $U$ be a Borel set in $U_i$. We have $U-\bv\subseteq V_i$.
Then from Lemma \ref{transverse-invariant-invariance} and Lemma
\ref{transverse-invariant-independence} we get
\begin{eqnarray*}
\mu(U-\bv) &=& \nu\otimes\lambda_d(f_i(U-\bv))\\
               &=&\nu\otimes\lambda_d(\psi^{(\b0,\bv)}(h_i(U)))\\
               &=&\nu\otimes\lambda_d(h_i(U)))\\
               &=&\mu_i(U)\\
               &=&\mu(U).
\end{eqnarray*}
This shows that $\mu$ is invariant.

\medskip

Let $C\in\B(\Gamma)$ and $\Theta\subseteq B_{\eta}(\b0)$ an open
set. We can assume that $C+\Theta$ is the disjoint union of sets
$C_i+\Theta$, where $C_i\subseteq C_{P_i}$ and $\Theta\subseteq
\Theta_i$, for every $1\leq i\leq n$. We have
$$
\mu(C_i+\Theta)=\nu^T\otimes\lambda_d(C_i\times
\Theta)=\nu(C_i)\lambda_d(\theta).
$$
Hence $\mu(C+\Theta)=\nu(C)\lambda_d(\Theta)$, which implies that
$\mu^T=\nu$.
\end{proof}

\begin{theorem}
\label{correspondence} There is a linear one-to-one correspondence
between the set of $\sigma$-finite invariant measures  and the set
of $\sigma$-finite transverse measures of $(X,\RR^d)$.
\end{theorem}
\begin{proof}
Lemma \ref{transverse-to-invariant} shows that the function that
associates to every invariant measure $\mu$ its transverse measure
$\mu^T$ is onto.

Let $\nu$ be a transverse measure and let $\mu$ be an invariant
measure such that $\mu^T=\nu$. Let $\{U_i,h_i\}_{i=1}^n$ be an
atlas of $X$. The measure $\mu\circ h_i^{-1}$ is defined on
$h_i(U_i)$ and verifies $$\mu\circ
h_i^{-1}(C\times\Theta)=\nu\otimes\lambda_d(C\times \Theta),$$ for
every $C\times \Theta\in (\B(\Gamma)\times \B(\RR^d))\cap
h_i(U_i)$. The  uniqueness of the product measure implies that
$\mu$ is the invariant measure of Lemma
\ref{transverse-to-invariant}. Therefore, the function that
associates to every invariant measure $\mu$ its transverse measure
$\mu^T$ is one-to-one. The linearity is clear.
\end{proof}

\end{document}